\documentclass[11pt]{article}
\usepackage{amsmath,amssymb,amsthm, latexsym,color,epsfig,enumerate,a4, graphicx}
\usepackage{xspace}
\usepackage{bbding}
\usepackage{tikz}
\usepackage{ifthen}
\usepackage{multicol}
\usepackage[T1]{fontenc}
\usetikzlibrary{calc,decorations.pathreplacing,decorations.pathmorphing,decorations.markings,shapes}

\newcounter{sarrow}
\newcommand\xrsquigarrow[1]{%
	\stepcounter{sarrow}
	\mathrel{
		\begin{tikzpicture}[decoration=snake]		
			\node[minimum size=1cm] (\thesarrow) {\strut#1};
			\draw[->,decorate] (\thesarrow.south west) -- (\thesarrow.south east);
		\end{tikzpicture}%
	}
}

\theoremstyle{plain}
\newtheorem{theorem}{Theorem}[section]
\newtheorem{lemma}[theorem]{Lemma}
\newtheorem{claim}[theorem]{Claim}

\newtheorem{corollary}[theorem]{Corollary}

\newtheorem{remark}[theorem]{Remark}

\newtheorem*{lemma*}{Lemma}

\theoremstyle{definition}
\newtheorem{definition}[theorem]{Definition}

\newcommand{\eps}{\ensuremath{\varepsilon}}

\newcommand{\Gnp}{\ensuremath{\mathcal G(n,p)}}
\newcommand{\Gknp}{\ensuremath{\mathcal G^{(k)}(n,p)}}
\newcommand{\Gknpp}{\ensuremath{\mathcal G^{(k + 1)}(n,p)}}

\newcommand{\Pkl}{\ensuremath{\mathrm{CP}^k_{\ell}}}
\newcommand{\Hkl}{\ensuremath{\mathrm{H}^k_{\ell}}}
\newcommand{\Kkl}{\ensuremath{\mathrm{P}^k_{\ell}}}
\newcommand{\Jkl}{\ensuremath{\mathrm{B}^k_{\ell}}}
\newcommand{\HJkl}{\ensuremath{\mathrm{BH}^k_{\ell}}}
\newcommand{\Wkl}{\ensuremath{W^k_{\ell}}}

\newcommand{\calE}{\ensuremath{\mathcal E}}

\date{}
\title{\vspace{-0.7cm}Powers of Hamilton cycles in random graphs and tight Hamilton cycles in random hypergraphs}
\author{
Rajko Nenadov$^\star$ 
\and
Nemanja \v Skori\'c\thanks{Institute of Theoretical Computer Science, ETH
Z\"urich, 8092 Z\"urich, Switzerland. \newline Email: {\tt
\{rnenadov|nskoric\}@inf.ethz.ch}.} 
}

\begin{document}
\maketitle

\begin{abstract}
    We show that for every $k \in \mathbb{N}$ there exists $C > 0$ such that if $p^k \ge C \log^8 n / n$ then asymptotically almost surely the random graph $\Gnp$ contains the $k$\textsuperscript{th} power of a Hamilton cycle. This determines the threshold for appearance of the square of a Hamilton cycle up to the logarithmic factor, improving a result of K\"uhn and Osthus. Moreover, our proof provides a randomized quasi-polynomial algorithm for finding such powers of cycles. Using similar ideas, we also give a randomized quasi-polynomial algorithm for finding a tight Hamilton cycle in the random $k$-uniform hypergraph $\Gknp$ for $p \ge C \log^8 n/ n$.
    
    The proofs are based on the absorbing method and follow the strategy of K\"uhn and Osthus, and Allen et al. The new ingredient is a general Connecting Lemma which allows us to connect tuples of vertices using arbitrary structures at a nearly optimal value of $p$. Both the Connecting Lemma and its proof, which is based on Janson's inequality and a greedy embedding strategy, might be of independent interest. 
\end{abstract}

\section{Introduction}

In this paper we consider the binomial random graph model $\Gnp$: given $n \in \mathbb{N}$ and a parameter $p = p(n) \in [0, 1]$, a graph $G \sim \Gnp$ on $n$ vertices is formed by adding each possible edge with probability $p$, independently of all other edges. One of the fundamental problems in the theory of random graphs is to determine for which values of $p$ does $\Gnp$ \emph{asymptotically almost surely}\footnote{We say that an event happens asymptotically almost surely if the probability of the event approaches $1$ as $n$ goes to infinity.} (\emph{a.a.s} for short) contain a desired subgraph $H$? One of the most prominent cases is where $H$ is a Hamilton cycle, for which this question was raised by Erd\H{o}s and R\'enyi \cite{ErRe60} and answered by P\'osa \cite{Po76} and Kor\v sunov \cite{Ko76} (see also \cite{AKS85,Bo84,komlos1983limit} for further developments). Here we study a generalisation to \emph{powers} of Hamilton cycles.

Inspired by the Po\'sa-Seymour conjecture on the existence of the $k$\textsuperscript{th} power of a Hamilton cycle (a \emph{Hamilton $k$-cycle} for short) in graphs with large minimum degree, K\"uhn and Osthus \cite{KuOs12} studied the appearance of Hamilton $k$-cycles in random graphs. Recall that the k$\textsuperscript{th}$ power of a graph $H$ is obtained by adding an edge between every two vertices of $H$ which are at distance at most $k$. 
K\"uhn and Osthus observed that a result of Riordan \cite{Rio00} implies that for (constant) $k \ge 3$ and $p \gg n^{-1/k}$ the random graph $\Gnp$ a.a.s contains a Hamilton $k$-cycle. Moreover, they proved that $p \ge n^{-1/2 + \varepsilon}$ suffices in the case of the square of a Hamilton cycle, for any constant $\varepsilon > 0$ (in this case Riordan's result only gives $p \gg n^{-1/3}$, the same as for $k = 3$). On the other hand, a simple first-moment argument shows that for $p < ((e - \varepsilon)/n)^{1/k}$ the random graph a.a.s does not contain a Hamilton $k$-cycle. In particular, for $k = 2$ there remains a gap of order $n^{\eps}$ in the range of $p$ covered by these results. 

Once we know that $\Gnp$ is likely to contain a certain subgraph, the next question is how fast can we find it? From such an algorithmic point of view Riordan's proof leaves us in complete dark -- it is fully based on the second-moment method and does not give any information on how to find the desired power of a Hamilton cycle. Contrary to that, the result of K\"uhn and Osthus requires somewhat larger value of $p$ but in turn gives a randomized polynomial time algorithm. Our first result is `sandwiched' between the two: the required bound on $p$ is at most a logarithmic factor away from the  optimal one and we obtain an algorithm with \emph{quasi-polynomial} running time.


\begin{theorem} \label{thm:main_graph}
    Let $k \ge 2$ be an integer. There exists $C > 0$ such that if $p^k \ge C \log^8 n / n$ then $\Gnp$ a.a.s contains a Hamilton $k$-cycle. Moreover, there exists a randomized algorithm with $n^{O(\log n)}$ running time which a.a.s finds such a Hamilton $k$-cycle.
\end{theorem}

In particular, Theorem \ref{thm:main_graph} determines a threshold for the existence of the square of a Hamilton cycle up to the logarithmic factor. Recently, Bennett, Dudek and Frieze \cite{bennett2016square} announced the improvement to $p \gg 1/\sqrt{n}$ using a second-moment method. 


Our second result gives a similar statement for the random $k$-uniform hypergraph model $\Gknp$, defined analogously to $\Gnp$. We say that a $k$-uniform hypergraph $G = (V, \calE)$ (a \emph{$k$-graph} for short) contains a \emph{tight Hamilton cycle} if there exists an ordering $(v_0, \ldots, v_{n-1})$ of $V$ such that $\{v_i, v_{i + 1}, \ldots, v_{i + k - 1}\} \in \calE$ for every $i \in \{0, \ldots, n - 1\}$ (where index addition is modulo $n$). 
Using the second-moment method, Dudek and Frieze \cite{dudek2013tight} have shown that $1/n$ is a threshold for the appearance of a tight Hamilton cycle in a random $k$-graph for any $k \ge 3$. 
An algorithmic proof for $p \ge n^{-1 + \varepsilon}$ was given by Allen, B\"ottcher, Kohayakawa and Person \cite{allen2013tight}, for any constant $\varepsilon > 0$. The next theorem, again, lies between these two results: the bound on $p$ is a logarithmic factor away from the optimal one and we obtain a quasi-polynomial algorithm (instead of a polynomial one which requires $p$ to be a factor of $n^{\eps}$ away from the smallest possible value, as in \cite{allen2013tight}). 

\begin{theorem}
    \label{thm:main_hypergraph}
    Let $k \ge 2$ be an integer. There exists $C > 0$ such that if $p \ge C \log^8 n/n$ then $\Gknp$ a.a.s contains a tight Hamilton cycle. Moreover, there exists a randomized algorithm with $n^{O(\log n)}$ running time which a.a.s finds such a cycle.
\end{theorem}

The paper is organized as follows. In the next subsection we briefly describe the \emph{absorbing} method. In Section \ref{sec:definitions} we introduce some definitions needed in the later parts of the paper. In Section \ref{sec:conn_lemma} we prove a Connecting Lemma, our main technical tool. It is usual for proofs using the absorbing method to have some form of the Connecting Lemma. The presented proof differs from ones found in the literature and is based on Janson's inequality and a simple greedy embedding strategy. We believe ideas used in this proof are of independent interest. 
In Section \ref{sec:absorbers} we then define \emph{absorbers} and prove their existence in random (hyper)graphs. Finally, in Section \ref{sec:main} we put everything together and prove Theorems \ref{thm:main_graph} and \ref{thm:main_hypergraph}. In Section \ref{sec:remarks} we discuss some further research direction.

\subsubsection*{Notation}
Given a (hyper)graph $G = (V, E)$, we denote by $v(G)$ and $e(G)$ the size of the vertex and the edge set, respectively. Given (hyper)graphs $H, G$ and a function $f \colon V(H) \rightarrow V(G)$, we say that $f$ is an \emph{embedding} if it is a bijection and $f(e) \in E(G)$ for every $e \in H$. Moreover, if additionally $f(e) \in E(G)$ if and only if $e \in E(H)$ then $f$ is an \emph{isomorphism}. 

Given a set $V$ and $k \in \mathbb{N}$, we denote by $V^k$ the family of all `distinct' $k$-tuples in $V$, that is,
$$
	V^k := \{(v_1, \ldots, v_k) \colon v_1, \ldots, v_k \in V \text{ and } v_i \neq v_j \text{ for } i \neq j\}.
$$
We usually denote an element of $V^k$ with a bold small letter. Given $\mathbf{a} = (v_1, \ldots, v_k) \in V^k$, we denote by $\overline{\mathbf{a}}$ the \emph{reversed} $k$-tuple $(v_k, \ldots, v_1)$. For all set-theoretic notions, we treat $k$-tuples as sets. Thus, for example, we say that two $k$-tuples $\mathbf{a} = (a_1, \ldots, a_k), \mathbf{b} = (b_1, \ldots, b_k) \in V^k$ are \emph{disjoint} if $\{a_1, \ldots, a_k\} \cap \{b_1, \ldots, b_k\} = \emptyset$.


 
Throughout the paper we use the term \emph{threshold} to denote the coarse threshold and logarithm is always used with base two, unless stated otherwise. We use the standard asymptotic notation $O$, $o$, $\Omega$ and $\omega$. Furthermore, given two functions $a$ and $b$ we write $a \ll b$ if $a = o(b)$, and $a \gg b$ if $a = \omega(b)$.

\subsection{Overview of the absorbing method} 
\label{sec:overview_absorber}

Proofs of Theorem \ref{thm:main_graph} and \ref{thm:main_hypergraph} are based on the \emph{absorbing} method. Even though the proof of R\"odl, Rucinski and Szemer\'edi \cite{rodl2006dirac} is usually attributed of being the first to use the absorbing method, it was already used (though without naming it) by Erd\H{o}s, Gy\'arf\'as and Pyber \cite{erdos1991vertex}. Since then the method has found numerous applications in (random) graph theory (for example, see \cite{allen2013tight,ferber2015robust,krivelevich1997triangle,KuOs12,levitt2010avoid}). Here we closely follow arguments given in \cite{ferber2015robust,KuOs12}.

We demonstrate the basic idea in the simpler case of Hamilton cycles in graphs. Let $A$ be a graph and $a, b \in V(A)$ two distinct vertices. Given a subset $X \subseteq V(A)$, we say that $A$ is an \emph{$(a, b, X)$-absorber} if for every subset $X' \subseteq X$ there exists a path $P_{X'} \subseteq A$ from $a$ to $b$ such that $V(P) = V(A) \setminus X'$. Let now $G = (V, E)$ be a graph in which we want to find a Hamilton cycle and suppose there exists a large subset $X \subseteq V(G)$ and an $(a, b, X)$-absorber $A \subseteq G$, for some vertices $a, b \in V(A)$. An important observation is that if $G$ contains a path $P$ from $a$ to $b$ such that
$$
V \setminus V(A) \subseteq V(P) \qquad \text{ and } \qquad V(P) \cap (V(A) \setminus X) = \{a, b\},
$$
that is, $P$ uses all the vertices in $V \setminus V(A)$, no vertex from $V(A) \setminus X$ (except $\{a,b\}$) and some (potentially none) vertices from $X$, we are done. Indeed, let us denote by $X'$ the subset of used vertices from $X$ in such a path. Then by the definition of the absorber there exists a path $P_{X'} \in A$ from $a$ to $b$ and $V(P) = V(A) \setminus X'$, completing the Hamilton cycle (see Figure \ref{fig:proof_idea}). 

\begin{figure}[!ht] 
    \label{fig:proof_idea}      
    \centering
    \begin{tikzpicture}[scale = 0.6]
        \tikzstyle{blob} = [fill=black,circle,inner sep=1.5pt,minimum size=0.5pt]
        \tikzstyle{small} = [fill=black,circle,inner sep=1pt,minimum size=0.5pt]        
        
        \draw[rounded corners] (15, 10) -- (15, 0) -- (0, 0) -- (0, 10) -- (15, 10); \node at (7.5, -0.5) {\Large $G$};
        \draw[rounded corners] (8, 10) -- (8, 6) -- (15, 6); \node at (14.7, 9.7) {\large $A$};
        
        \draw[rounded corners] (8, 8.5) -- (12, 8.5) -- (12, 6); \node at (11.7, 8.1) {$X$};

        \draw[blue, line width=0.5mm] plot [smooth, tension=1]  coordinates 
        {
            (8.5, 9.5) (5, 9) (1.5, 9.3)  (1.5, 7) (7.5, 8) (7, 6.5) (1, 6.5) (2, 5) (8, 6.5) (7, 5) (1, 4) (2, 3) (8, 5) (7, 3) (1, 2) (2, 1)
            (5, 2) 
            (9, 1) (9, 7) (10, 5) (10, 1) (11, 2) (11, 5) (12, 4) (12, 1) (13, 1) (13, 5) (14, 4) (14, 1) (14.5, 2) (14, 6.5)
        };
        \fill[white] (8.1,6.1) rectangle (10,8);
        \draw[blue, thick, dashed] plot [smooth, tension=1]  coordinates 
        {
            (8.5, 9.5) (5, 9) (1.5, 9.3)  (1.5, 7) (7.5, 8) (7, 6.5) (1, 6.5) (2, 5) (8, 6.5) (7, 5) (1, 4) (2, 3) (8, 5) (7, 3) (1, 2) (2, 1)
            (5, 2) 
            (9, 1) (9, 7) (10, 5) (10, 1) (11, 2) (11, 5) (12, 4) (12, 1) (13, 1) (13, 5) (14, 4) (14, 1) (14.5, 2) (14, 6.5)
        };
        
        \node at (4.7, 1.5) {\textcolor{blue}{$P$}};
        \node at (8.6, 6.8) {\textcolor{blue}{\small $X'$}};
        
        \node at (11, 9.65) {\textcolor{red}{$P_{X'}$}};
        
        \draw[red, thick] plot [smooth, tension=1]  coordinates 
        {
            (8.5, 9.5) (8.7, 8) (9, 9) (9.5, 8) (10, 9.5) (10.5, 7) (11, 9.3) (11.5, 6.5)
            (12.7, 9.8) (13.3, 6.3) (13.7, 9) (14.5, 9) (14, 6.5)
        };
        
        \node[blob] at (8.5, 9.5) {}; \node at (9, 9.5) {$a$};
        \node[blob] at (14, 6.5) {}; \node at (14.5, 6.5) {$b$};

    \end{tikzpicture}
    \caption{Constructing a Hamilton cycle using the absorber $A$.}
\end{figure}
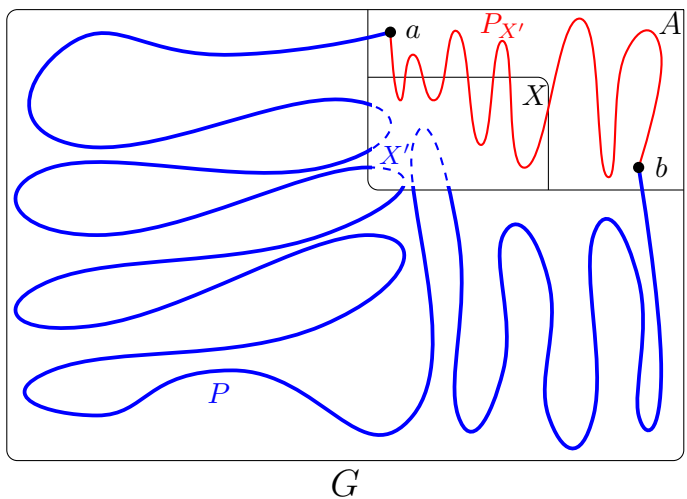

A priori it is not clear why this helps. However, it turns out that by following the described approach we split the problem into two easier ones. The first problem is to find an absorbing structure which is usually not larger than $\varepsilon n$. The second one is to find a path with prescribed endpoints which additionally covers all vertices from the set $V \setminus V(A)$ and an arbitrary subset of vertices from $X$. If $X$ is sufficiently large, this roughly corresponds to finding an almost-spanning path which is known to be a much easier problem.

Going from Hamilton cycles in graphs to Hamilton $k$-cycles in graphs and tight Hamilton cycles in $k$-graphs requires more careful definition of absorbers, but the proof idea is essentially the same as described. The main technical tool is a Connecting Lemma given in Section \ref{sec:conn_lemma}. Then, in the first step we find absorbers using Janson's inequality and the Connecting Lemma (Section \ref{sec:absorbers}). In the second step we find a desired  almost-spanning path using matchings in bipartite graphs and, again, the Connecting Lemma (Section \ref{sec:main}). 
The absorbers we use here stem from absorbers used in \cite{ferber2015robust,KuOs12}.

\section{Definitions of some graphs and hypergraphs} \label{sec:definitions}

The following graphs and hypergraphs are used often throughout the paper, thus we give their definitions here for easier reference:

\begin{itemize}
	\item \textbf{$(k, \ell)$-path} $\Kkl$: the $k$\textsuperscript{th} power of a path on $\ell$ vertices. More precisely, $\Kkl$ is the graph on the vertex set $\{u_1, \ldots, u_\ell\}$ and the edge set consisting of all $\{u_i, u_j\}$ such that $0 < j - i \le k$. 
	      
	      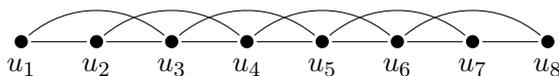
\begin{figure}[!ht]
	      	\centering
	      	\begin{tikzpicture}[%
	      			>=stealth,
	      			shorten >=1pt,
	      			shorten <=1pt
	      		]
	      		\tikzstyle{vertex} = [circle, fill, inner sep=0pt, minimum size=5pt]
	      		\tikzstyle{edge} = [-]			
	      		\tikzstyle{top} = [bend left=40]
	      		\tikzstyle{bot} = [bend right=40]
	      		
	      		\foreach \i in {1,...,8}{
	      			\node (v\i) [vertex, label=below:$u_\i$] at (\i,0) {};
	      		}
	      		
	      		\foreach \i in {1, ..., 7}{
	      			\draw[edge] let \n1={int(mod(\i+1,10))} in (v\i) to (v\n1);
	      		}
	      		
	      		\foreach \i in {1, ..., 6}{
	      			\draw[edge] let \n1={int(mod(\i+2,10))} in (v\i) edge [top] (v\n1);
	      		}				
	      	\end{tikzpicture}
	      	\caption{The $(2,8)$-path $\mathrm{P}_8^2$.}		
	      	\label{fig:k_path}
	      \end{figure}

	\item \textbf{$(k, \ell)$-connecting-path} $\Pkl$: the graph obtained from the $k$-path on $\ell$ vertices by removing the edges with both endpoints in either the first or the last $k$ vertices (see Figure \ref{fig:k_connecting}). More precisely, $\Pkl$ is the graph on the vertex set $\{u_1, \ldots, u_{\ell}\}$ and the edge set consisting of all $\{u_i, u_j\}$ such that $0 < j - i \le k$ and $|\{i, j\} \cap \{k+1, \ldots, \ell - k\}| \ge 1$. 
	      
	      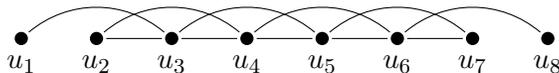
\begin{figure}[!ht]
	      	\centering
	      	\begin{tikzpicture}[%
	      			>=stealth,
	      			shorten >=1pt,
	      			shorten <=1pt
	      		]
	      		\tikzstyle{vertex} = [circle, fill, inner sep=0pt, minimum size=5pt]
	      		\tikzstyle{edge} = [-]			
	      		\tikzstyle{top} = [bend left=40]
	      		\tikzstyle{bot} = [bend right=40]
	      		
	      		\foreach \i in {1,...,8}{
	      			\node (v\i) [vertex, label=below:$u_\i$] at (\i,0) {};
	      		}
	      		
	      		\foreach \i in {2,...,6}{
	      			\draw[edge] let \n1={int(mod(\i+1,10))} in (v\i) to (v\n1);
	      		}
	      		
	      		\foreach \i in {1, ..., 6}{
	      			\draw[edge] let \n1={int(mod(\i+2,10))} in (v\i) edge [top] (v\n1);
	      		}				
	      	\end{tikzpicture}
	      	\caption{The $(2,8)$-connecting-path $\mathrm{CP}_8^2$.}		
	      	\label{fig:k_connecting}
	      \end{figure} 
	      
	\item \textbf{$(k, \ell)$-tight-path} $\Hkl$: the $(k+1)$-graph with the vertex set $\{u_1, \ldots, u_{\ell}\}$ and the edge set consisting of $\{u_i, \ldots, u_{i+k}\}$ for all $1 \le i \le \ell - k$.
	      
	      \begin{figure}[!ht]
	      	\centering
	      	\begin{tikzpicture}[%
	      			>=stealth,
	      			shorten >=1pt,
	      			shorten <=1pt
	      		]
	      		\tikzstyle{vertex} = [circle, fill, inner sep=0pt, minimum size=5pt]
	      		\tikzstyle{edge} = [-]			
	      		\tikzstyle{top} = [bend left=40]
	      		\tikzstyle{bot} = [bend right=40]
	      		
	      		\foreach \i in {1,...,8}{
	      			\node (v\i) [vertex] at (\i,0) {};
	      			\node at (\i, -0.6) {$u_\i$};
	      		}
	      		
	      		\foreach \i in {2, ..., 7}{
	      			\draw (\i,0) ellipse (1.2cm and 0.4cm);				
	      		}
	      	\end{tikzpicture}
	      	\caption{The $(2,8)$-tight-path $\mathrm{H}_8^2$.}		
	      	\label{fig:k_tight_path}
	      \end{figure}
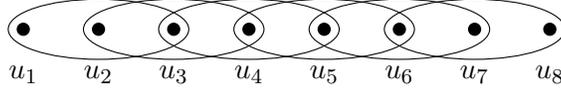 
\end{itemize}

In case we are not interested in the length, we simply write $k$-path, $k$-connecting-path and $k$-tight-path. 

Given a graph $G$ and disjoint $r$-tuples of vertices $\mathbf{a} = (a_1, \ldots, a_r)$ and $\mathbf{b} = (b_1, \ldots, b_r)$, a $k$-path from $\mathbf{a}$ to $\mathbf{b}$ is any subgraph $P \subseteq G$ such that there exists an isomorphism $f \colon P \rightarrow \Kkl$, for some $\ell \ge 2k$ (unless we specify the length), with $f((u_1, \ldots, u_r)) = \mathbf{a}$ and $f((u_{\ell - r + 1}, \ldots,u_\ell)) = \mathbf{b}$. A $k$-connecting-path and $k$-tight-path from $\mathbf{a}$ to $\mathbf{b}$ are defined analogously.

Given a $k$-path $P$, we denote by $\overline{P}$ the $k$-path obtained by traversing $P$ in the opposite direction. We define $\overline{P}$ analogously in case $P$ is a $k$-connecting-path or $k$-tight-path. Thus if $P$ is a $k$-path from $\mathbf{a}$ to $\mathbf{b}$, then $\overline{P}$ is a $k$-path from $\overline{\mathbf{b}}$ to $\overline{\mathbf{a}}$.

\section{Connecting Lemma} \label{sec:conn_lemma}


Given (hyper)graphs $G, F$ and $r$-tuples of vertices $\mathbf{x} \in V(F)^r, \mathbf{y} \in V(G)^r$, for some $0 \le r \le v(F)$, we say that $F' \subseteq G$ is an \emph{$(F, \mathbf{x}, \mathbf{y})$-copy} if there exists an isomorphism $f \colon F \rightarrow F'$ such that $f(\mathbf{x}) = \mathbf{y}$. In other words, an $(F, \mathbf{x}, \mathbf{y})$-copy `connects' the $r$-tuples  of vertices with the graph $F$ in some prescribed way. We call vertices $f(F) \setminus \mathbf{y}$ the \emph{internal vertices}. Using this terminology, a $k$-path from $\mathbf{a}$ to $\mathbf{b}$ in $G$, for some $r$-tuples $\mathbf{a}, \mathbf{b} \in V(G)^r$, is a $(\Kkl, \mathbf{u}, (\mathbf{a}, \mathbf{b}))$-copy for $\mathbf{u} = (u_1, \ldots, u_r, u_{\ell - r + 1}, \ldots, u_\ell)$ (see the previous section for notation).

Informally, the main result of this section, Lemma \ref{lemma:connecting}, considers the existence of a family of vertex-disjoint $(F, \mathbf{x}, \mathbf{y}_i)$-copies for some given family of $t$-many $r$-tuples $\{\mathbf{y}_i\}_{i \in [t]}$. The following definition makes this precise.

\begin{definition}
	Let $t \in \mathbb{N}$ be an integer, let $G, F$ be (hyper)graphs and $\mathbf{x} \in V(F)^r$ an $r$-tuple of vertices, for some $0 \le r \le v(F)$. Given a family $\mathcal{Y} = \{ \mathbf{y}_i \in V(G)^r \}_{i \in [t]}$ of $r$-tuples, we say that a collection $\{F_i \subseteq G\}_{i \in [t]}$ of subgraphs of $G$ forms an \emph{$(F, \mathbf{x}, \mathcal{Y})$-matching} if the following holds:
	\begin{itemize}
		\item $F_i$ is an $(F, \mathbf{x}, \mathbf{y}_i)$-copy for every $i \in [t]$, and
		\item $V(F_i) \cap V(F_j) = \emptyset$ for all $i \neq j \in [t]$.	
	\end{itemize}
\end{definition}

In other words, an $(F, \mathbf{x}, \mathcal{Y})$-matching `connects' prescribed tuples of vertices from $\mathcal{Y}$ with disjoint copies of $F$. The following lemma concerns the existence of such matchings in random graph and hypergraphs. 

\begin{definition}[$m(F, X)$-density] \label{def:density}
	Given a (hyper)graph $F$ and a proper (possibly empty) subset $X \subset V(F)$, we define the \emph{$m(F, X)$-density} as follows,
	$$
	m(F, X) = \max_{\substack{F' \subseteq F\\e(F')>0}} \left\{ \frac{e(F')}{v(F') - \max\{1, |V(F') \cap X|\}} \colon
	\text{either } X \subseteq V(F') \text{ or } X \cap V(F') = \emptyset \right\}.
	$$ 
\end{definition}


\begin{lemma}[Connecting Lemma] \label{lemma:connecting}
	Given integers $k \ge 2$ and $r \ge 0$ and positive $K \in \mathbb{R}$, there exist $C > 1$ and $n_0 \in \mathbb{N}$ such that the following holds for every $n \ge n_0$. Let $F$ be a $k$-graph with $v(F) \ge r + 2$ vertices (note that $F$ can depend on $n$), $\mathbf{x} \in V(F)^r$ an $r$-tuple of independent vertices in $F$ and $W \subseteq [n]$ a subset of size $|W| \ge C v(F)^4 \log^2 n$. If 
	$$
	p \ge \left( \frac{C v(F)^4 \log^2 n}{|W|} \right)^{1/m(F, \mathbf{x})}
	$$
	then $G \sim \Gknp$ has the following property with probability at least $1 - n^{-K}$: 
					
	For every $t \in \mathbb{N}$ such that $t (v(F) - r) \le |W|/4$ and a family $\mathcal{Y} = \{\mathbf{y}_i \in (V(G) \setminus W)^t\}_{i \in [t]}$ of disjoint $r$-tuples, there exists an $(F, \mathbf{x}, \mathcal{Y})$-matching in $G$ with all internal vertices being in $W$. Moreover, there exists an algorithm which finds such a $(F, \mathbf{x}, \mathcal{Y})$-matching in $n^{O(v(F))}$ time.
\end{lemma}

The proof of the Connecting Lemma relies on the following claim which we prove in Section \ref{sec:proof_of_expansion}.  

\begin{claim} \label{claim:expansion_F_matching}
	Given integers $k \ge 2$ and $r \ge 0$ and positive $K \in \mathbb{R}$, there exist $C > 1$ and $n_0 \in \mathbb{N}$ such that the following holds for every $n \ge n_0$. Let $F$ be a $k$-graph with $v(F) \ge r + 2$ vertices, $\mathbf{x} \in V(F)^r$ an $r$-tuple of independent vertices in $F$ and $S \subseteq [n]$ a subset of size $|S| \ge C v(F)^4 \log n$. If 
	$$
	p \ge \left( \frac{C v(F)^4 \log n}{|S|} \right)^{1/m(F, \mathbf{x})}
	$$
	then $G \sim \Gknp$ has the following property with probability at least $1 - n^{-K}$: 	

	For every $t \in \mathbb{N}$ such that $t (v(F) - r) \le |S|/2$, every  subset $D \subseteq S$ of size $|D| \le t (v(F) - r)$, and every  family $\{\mathbf{y}_i \in (V(G) \setminus S)^r\}_{i \in [t]}$ of disjoint $r$-tuples, 
	there exists an $(F, \mathbf{x}, \mathbf{y}_i)$-copy $F_i \subseteq G$, for some $i \in [t]$, with $V(F_i) \setminus \mathbf{y}_i \subseteq S \setminus D$.
\end{claim}

We now use Claim \ref{claim:expansion_F_matching} to derive the Connecting Lemma.

\begin{proof}[Proof of the Connecting Lemma (Lemma \ref{lemma:connecting})]
	Let $C'$ be a constant given by Claim \ref{claim:expansion_F_matching} for $k$ and $K+1$ (as $K$). We prove the lemma with $C = 2C'$, that is
	$$
	p \ge \left(\frac{2 C' v(F)^4 \log^2 n}{|W|} \right)^{1/m(F, \mathbf{x})}
	$$
	where $W$ is a given subset. Let $W_1, \ldots, W_{\log n} \subseteq W$ be disjoint subsets such that each $W_i$ is of size $|W_i| = \max\{|W| / 2^{i+1}, |W| / 2\log n\}$. Then $G \sim \Gknp$ a.a.s satisfies the property of Claim \ref{claim:expansion_F_matching} for every $W_i$ (as $S$). This follows from a simple union bound over all $W_i$'s and the fact that for a particular one the property holds with probability at least $1 - 1 / n^{K+1}$. We show that such $G$ contains an $(F, \mathbf{x}, \mathcal{Y}$)-matching for every (valid) family $\mathcal{Y}$.
					
	To this end, consider a family of disjoint $r$-tuples $\mathcal{Y} = \{\mathbf{y}_i \in (V(G) \setminus W)^{r}\}_{i \in [t]}$, for some $t \in \mathbb{N}$ such that $t(v(F) - r) \le |W|/4$. We obtain a desired $(F, \mathbf{x}, \mathcal{Y})$-matching in $\log n$ rounds. Set $R_0 := [t]$ and in each round $1 \le j \le \log n$ greedily construct an $(F, \mathbf{x}, \mathcal{Y}_j)$-matching with all internal vertices in $W_j$ as follows: set $\mathcal{F}_j := \emptyset$ and $I_j := \emptyset$ and as long as there exists an $(F, \mathbf{x}, \mathbf{y}_i)$-copy $F_i \subseteq G$, for some $i \in R_{j-1} \setminus I_j$, which is vertex-disjoint from all copies in $\mathcal{F}_j$ and has all internal vertices in $W_j$, add $F_i$ to $\mathcal{F}_j$ and $i$ to $I_j$. Once there is no such $i \in R_{j-1} \setminus I_j$, set $R_j := R_{j-1} \setminus I_j$ and proceed to the next round. Note that $\mathcal{F}_j$ forms an $(F, \mathbf{x}, \mathcal{Y}_j)$-matching for $\mathcal{Y}_j = \{\mathbf{y}_i\}_{i \in I_j}$. Furthermore, as in each round we only use vertices from $W_j$ as internal ones, matchings constructed in different rounds are clearly vertex-disjoint. Therefore, if by the end of the last round we matched all $r$-tuples, that is, $R_{\log n} = \emptyset$, then we have an $(F, \mathbf{x}, \mathcal{Y})$-matching. We show that $|R_j| \le t / 2^j$ after every step $j$, which implies the desired conclusion as $t < n$.
					
	Suppose towards a contradiction that after some round $j$ we have $|R_j| > t / 2^j$ and, furthermore, suppose $j \ge 1$ is the first such round. Note that this implies $|R_{j-1}| \le t / 2^{j-1}$ and therefore $|I_j| < t / 2^j$. Moreover, the $(F, \mathbf{x}, \mathcal{Y}_j)$-matching $\mathcal{F}_j$ has exactly $|I_j| (v(F) - r)$ internal vertices in $W_j$, thus the set $D_j \subseteq W_j$ of `used' vertices is of size at most
	\begin{equation} \label{eq:W_j}
		|D_j| = |I_j| (v(F) - r) \le \frac{t}{2^j} (v(F) - r).
	\end{equation}
	On the other hand, as $t(v(F) - r) \le |W|/4$ we have
	\begin{equation} \label{eq:W_j2}
		\frac{t}{2^j} (v(F) - r) \le \frac{|W|}{4 \cdot 2^j} \le |W_j|/2.
	\end{equation}
	Pick an arbitrary subset $R' \subseteq R_j$ of size $|R'| = t / 2^j$. We can apply the property of Claim \ref{claim:expansion_F_matching} with $D_j$ (as $D$), $W_j$ (as $S$) and $\{\mathbf{y}_i\}_{i \in R'}$ (owing to \eqref{eq:W_j} and \eqref{eq:W_j2}) to conclude that there exists a $(F, \mathbf{x}, \mathbf{y}_i)$-copy with all internal vertices being in $W_j \setminus D_j$, for some $i \in R'$. As such copy is vertex-disjoint from all other copies in $\mathcal{F}_j$ we get a contradiction with the assumption that the procedure has finished the round $j$.
					
	By checking for each $i \in R_j$ whether it can be added to $I_j$, that is, whether there exists an $(F, \mathbf{x}, \mathbf{y}_i)$-copy with internal vertices in $W_j$ and which is vertex-disjoint from $\mathcal{F}_j$, can be done in $n^{O(v(F))}$ time by simply trying all possible choices for such a copy. As we finish in $\log n$ rounds and in each round we perform this step for $t / 2^j < n$ indices, this gives a simple $n^{O(v(F))}$ algorithm for constructing an $(F, \mathbf{x}, \mathcal{Y})$-matching.
\end{proof}

\subsection{Some useful corollaries of the Connecting Lemma} \label{sec:corollaries}

In this section we collect some corollaries used in the proof of our main theorems. The following corollary spells out the Connecting Lemma in case $r = 0$. Note that in this case the definition of $m(F, \emptyset)$-density coincides with the well-known \emph{$m_1$-density},
$$
m_1(F) = \max_{\substack{F' \subseteq F \\ e(F') > 0}}\{ e(F') / (v(F') - 1) \}.
$$

\begin{corollary} \label{cor:factor}
	For every integer $k \ge 2$, there exists $C > 0$ such that the following holds for every $k$-graph $F$ (which might depend on $n$) and a subset $W \subseteq [n]$ of size $|W| \ge C v(F)^4 \log^2 n$. If 
	$$
	p \ge \left( \frac{C v(F)^4 \log^2 n}{|W|} \right)^{1/m_1(F)}
	$$	
	then $G \sim \Gknp$ a.a.s contains a family of at least $|W| / 4v(F)$ pairwise vertex-disjoint copies of $F$ with all vertices being in $W$.
\end{corollary}

The following corollary is tailored for the proof of Theorem \ref{thm:main_graph}.

\begin{corollary}[Graph Connecting Lemma] \label{cor:connecting}
	Given an integer $k \ge 2$ and a subset $W \subseteq [n]$ of size $|W| \ge n / \log^3 n$, there exists a constant $C = C(k) > 0$ such that if 
	$$
	p \ge \left( \frac{C \log^6 n}{|W|} \right)^{1/k}
	$$
	then $G \sim \Gnp$ a.a.s has the following property: for every family of disjoint $2k$-tuples $\{(\mathbf{a}_i, \mathbf{b}_i) \in ([n] \setminus W)^{2k}\}_{i \in [t]}$ of size $t \le |W| / 4 \log n$, there exists a family $\{P_i \subseteq G\}_{i \in [t]}$ of vertex-disjoint $(k, \log n)$-connecting-paths such that
	\begin{enumerate}[(i)]
		\item $P_i$ is a $(k, \log n)$-connecting-paths from $\mathbf{a}_i$ to $\mathbf{b}_i$, and
		\item $V(P_i) \setminus \{\mathbf{a}_i \cup \mathbf{b}_i\} \subseteq W$.
	\end{enumerate}
	Moreover, such $k$-connecting-paths can be found in $n^{O(\log n)}$-time.
\end{corollary}
\begin{proof}
	As pointed out earlier, a subgraph $P_i \subseteq G$ is a $(k, \log n)$-connecting-path from $\mathbf{a}_i$ to $\mathbf{b}_i$ if and only if it is an $(\Pkl, \mathbf{u}, (\mathbf{a}_i, \mathbf{b}_i))$-copy, where $\mathbf{u} = (u_1, \ldots, u_k, u_{\ell - k + 1}, \ldots, u_\ell)$ and $\ell = \log n$ (see Section \ref{sec:definitions}). Therefore, it suffices to show that for every family $\mathcal{Y} = \{(\mathbf{a}_i, \mathbf{b}_i)\}_{i \in [t]}$ of $2k$-tuples satisfying given conditions there exists an $(\Pkl, \mathbf{u}, \mathcal{Y})$-matching with all internal vertices being in $W$. By Lemma \ref{lemma:connecting}, $G \sim \Gnp$ has such a property for
	$$
		p = \Omega \left(\log^6 n / |W| \right)^{1/m(\Pkl, \mathbf{u})},
	$$
	We show that $m(\Pkl, \mathbf{u}) \le k + 8 k^3 / \log n$. Note that this implies $1 / m(\Pkl, \mathbf{u}) \ge \frac{1}{k} - \frac{8k}{\log n}$ (with room to spare), and therefore
	\begin{align}
		\left( \log^6 n / |W| \right)^{1 / m(\Pkl, \mathbf{u})} &\le \left( \log^6 n / |W| \right)^{\frac{1}{k} - \frac{8k}{\log n}} \nonumber \\
		&=   \left( \log^6 n / |W| \right)^{k} \left( |W| / \log^6 n\right)^{8k / \log n} \nonumber \\
		&\le \left( \log^6 n / |W| \right)^{k} n^{8k / \log n} \nonumber \\
		& = O\left( \log^6 n / |W| \right)^{k},  \label{eq:plug_m}              
	\end{align}
	as desired.

	Consider some graph $F' \subseteq \Pkl$ such that $e(F') > 0$ and either $V(F') \cap \mathbf{u} = \emptyset$ or $\mathbf{u} \subseteq V(F')$. For each $i \in [\ell]$ we define $e_i$ and $\overline{e}_i$ as the number of edges $e \in F'$ such that $u_i \in e$ and $e \subseteq \{u_1, \ldots, u_{i-1}\}$ and $e \subseteq \{u_{i + 1}, \ldots, u_\ell\}$, respectively. 
	Furthermore, for a vertex $v \in F'$ let $i(v)$ denote its index with respect to the ordering $\{u_1, \ldots, u_\ell\}$, i.e. $v \equiv u_{i(v)}$. Finally, let $\{v_1, \ldots, v_s\}$ be an ordering of $V(F')$ such that $i(v_j) < i(v_{j+1})$ for all $j \in [s-1]$. 
	We use the following observation,
	\begin{equation} \label{eq:E_F_sum}
		e(F') = \sum_{j = 2}^{v(F)} e_{i(v_j)} = \sum_{j = 1}^{v(F) - 1} \overline{e}_{i(v_j)}.
	\end{equation}
					
	Let us first consider the case $\mathbf{u} \cap V(F') = \emptyset$. 
	From \eqref{eq:E_F_sum} we have
	$$
		\frac{e(F')}{v(F') - 1} \le \frac{k (v(F') - 1)}{v(F') - 1} = k.
	$$			
	Next, suppose $\mathbf{u} \subseteq V(F')$. If $v(F') \ge (\ell - 3k) / k$ then, again, from \eqref{eq:E_F_sum} we have
	$$
		\frac{e(F')}{v(F') - 2k} \le \frac{k (v(F') - 2k) + 2k^2}{v(F') - 2k} = k + \frac{2k^2}{v(F') - 2k} \le k + \frac{8k^3}{\log n}.
	$$ 
	In the last inequality we used $v(F') - 2k \ge v(F')/2$ and $\ell - 3k \ge \ell / 2$, which holds for sufficiently large $n$ (recall $\ell = \log n$).
	Finally, let us consider the case $v(F') < (\ell - 3k) / k$. Since $F'$ contains $\mathbf{u}$, we conclude that $\{u_1, \ldots, u_\ell\}$ and $\{u_{\ell - k +1}, \ldots, u_\ell\}$ must belong to different connected components of $F'$ as otherwise 
	$$
		v(F') \ge \lfloor (\ell - 2k) / k \rfloor > (\ell - 3k) / k,
	$$
	which contradicts our assumption. Let $F'_1$ be the connected component containing $\{u_1, \ldots, u_\ell\}$ and $F'_2$ the union of all the other components. As $e_i = 0$ for all $1 \le i \le k$ and $\overline{e}_i = 0$ for all $\ell - k + 1 \le i \le \ell$, from \eqref{eq:E_F_sum} we conclude $e(F_1') \le k ( v(F_1') - k)$ and $e(F_2') \le k (v(F_2') - k)$. Therefore,
	$$
		\frac{e(F')}{v(F') - 2k} = \frac{e(F_1') + e(F_2')}{v(F') - 2k} \le \frac{ k (v(F'_1) + v(F'_2) - 2k)}{v(F') - 2k} = k,
	$$
	and we conclude $m(F, \mathbf{u}) \le k + 8k^3 / \log n$.
\end{proof}

	

The following corollary is tailored for the proof of Theorem \ref{thm:main_hypergraph}. The statement is exactly the same as in Corollary \ref{cor:connecting} with $k$-connecting-paths replaced by $k$-tight-paths.

\begin{corollary}[Hypergraph Connecting Lemma] \label{cor:connecting_hyper}
	Given an integer $k \ge 2$ and a subset $W \subseteq [n]$ of size $|W| \ge n / \log^3 n$, there exists a constant $C = C(k) > 0$ such that if 
	$$
	p \ge \frac{C \log^6 n}{|W|}
	$$
	then $G \sim \Gknpp$ a.a.s has the following property: for every family of disjoint $2k$-tuples $\{(\mathbf{a}_i, \mathbf{b}_i) \subseteq ([n] \setminus W)^{2k}\}_{i \in [t]}$ of size $t \le |W| / 4 \log n$, there exists a family $\{P_i \subseteq G\}_{i \in [t]}$ of vertex disjoint $(k, \log n)$-tight-paths such that
	\begin{enumerate}[(i)]
		\item $P_i$ is a $(k, \log n)$-tight-path from $\mathbf{a}_i$ to $\mathbf{b}_i$, and
		\item $V(P_i) \setminus \{\mathbf{a}_i \cup \mathbf{b}_i\} \subseteq W$.
	\end{enumerate}
	Moreover, such $k$-tight-paths can be found in $n^{O(\log n)}$-time.
\end{corollary}
\begin{proof}
	Note that a subgraph $P_i \subseteq G$ is a $(k, \log n)$-tight-path from $\mathbf{a}_i$ to $\mathbf{b}_i$ if and only if it is a $(\Hkl, \mathbf{u}, (\mathbf{a}_i, \mathbf{b}_i))$-copy, where $\mathbf{u} = (u_1, \ldots, u_k, u_{\ell - k + 1}, \ldots, u_\ell)$ and $\ell = \log n$ (see Section \ref{sec:definitions}). Therefore, as in the proof of Corollary \ref{cor:connecting}, it suffices to show that for every family $\mathcal{Y} = \{(\mathbf{a}_i, \mathbf{b}_i)\}_{i \in [t]}$ satisfying conditions of the corollary there exists an $(\Hkl, \mathbf{u}, \mathcal{Y})$-matching with all internal vertices being in $W$. Following the same argument as in the proof of Corollary \ref{cor:connecting}, we obtain $m(\Hkl, \mathbf{u}) \le 1 + 8 k^2 / \log n$. Therefore, Lemma \ref{lemma:connecting} tells us that the desired property a.a.s holds for 
	$$
		p = \Omega(\log^6 n / |W|)^{1 / m(\Hkl, \mathbf{u})} = \Omega( \log^6 n / |W| ).
	$$
	(See \eqref{eq:plug_m} for details of the calculation).
\end{proof}

\subsection{Proof of Claim \ref{claim:expansion_F_matching}}
\label{sec:proof_of_expansion}

We use lower tail estimates for random
variables which count the number of copies of certain graphs in a
random graph. The following version of
Janson's inequality, tailored for graphs, will suffice. The
statement follows immediately from Theorems $8.1.1$ and $8.1.2$ in
\cite{alon2004probabilistic}.

\begin{theorem}[Janson's inequality] \label{thm:Janson}
	Let $k \ge 2$ be an integer, $p = p(n) \in (0, 1]$ and consider a family $\{ H_i \}_{i \in
		\mathcal{I}}$ of subgraphs of the complete (hyper)graph on the vertex set
	$[n]$. Let $G \sim \Gknp$ and, for each $i \in \mathcal{I}$, let $X_i$
	denote the indicator random variable for the event that $H_i \subseteq
	G$ and, for each ordered pair $(i, j) \in \mathcal{I} \times
	\mathcal{I}$ with $i \neq j$, write $H_i \sim H_j$ if $E(H_i)\cap
	E(H_j)\neq \emptyset$. Then, for
	\begin{align*}
		X      & = \sum_{i \in \mathcal{I}} X_i,                             \\
		\mu    & = \mathbb{E}[X] = \sum_{i \in \mathcal{I}} p^{e(H_i)},      \\
		\delta & = \sum_{\substack{(i, j) \in \mathcal{I} \times \mathcal{I} \\ H_i \sim H_j}} \mathbb{E}[X_i X_j] = \sum_{\substack{(i, j) \in \mathcal{I} \times \mathcal{I} \\ H_i \sim H_j}} p^{e(H_i) + e(H_j) - e(H_i \cap H_j)}
	\end{align*}
	and any $0 < \gamma < 1$, we have
	$$ \Pr[X < (1 - \gamma)\mu] \le e^{- \frac{\gamma^2 \mu^2}{2(\mu + \delta)}}. $$
\end{theorem}

The proof of Claim \ref{claim:expansion_F_matching} follows a straightforward but tedious argument using Janson's inequality.

\begin{proof}[Proof of Claim \ref{claim:expansion_F_matching}]
					
	Let $V(F) \setminus \mathbf{x} = \{u_1, \ldots, u_m\}$ be an arbitrary labeling of the vertices in $F \setminus \mathbf{x}$, where $m = v(F) - r$. In order to prove the claim it suffices to show that for a particular:
	\begin{itemize}
		\item $t \in \mathbb{N}$ such that $t (v(F) - r) n \le |S| / 2$,
		\item a subset $D \subseteq S$ of size $|D| =  t (v(F) - r)$, and 
		\item a family $\mathcal{Y} = \{\mathbf{y}_i \in (V(G) \setminus S)^r \}_{i \in [t]}$ of disjoint $r$-tuples, 	
	\end{itemize}
	with probability at least $1 - 2^{- (K+2)  t \cdot v(F) \log n}$ the random (hyper)graph $G \sim \Gknp$ contains a $(F, \mathbf{x}, \mathbf{y}_i)$-copy $F_i$, for some $i \in [t]$, such that $V(F_i) \setminus \mathbf{y}_i \subseteq S \setminus D$. Indeed, as there are at most $n$ choices for $t$, and for each $t$ at most 
	$$
	\binom{|S|}{|D|} = \binom{|S|}{t(v(F) - r)} \le 2^{t \cdot (v(F) - r) \log n}
	$$
	choices for $D$ and at most $n^{r t} \le 2^{r t \log n}$ choices for the family $\mathcal{Y}$, a union-bound over all such choices implies that the claim holds with probability at least
	$$
	1 - n 2^{ t \cdot v(F) \log n - (K + 2) t \cdot v(F) \log n} > 1 - n^{-K}.
	$$

	In the rest of the proof we show the desired probability for some chosen $t$, $D$ and $\mathcal{Y}$ as stated above. First, observe that we always have $|S \setminus D| \ge |S|/2$ and set	$S' := S \setminus D$. For each $i \in [t]$, let $\mathcal{F}_i$ denote the family of all \emph{valid lexicographical} $(F, \mathbf{x}, \mathbf{y}_i)$-copies in $K_n^{(k)}$, that is, $(F, \mathbf{x}, \mathbf{y}_i)$-copies $F' \subseteq K_n^{(k)}$ such that 
	\begin{enumerate}	
		\item $V(F') \setminus \mathbf{y}_i \subseteq S'$, and
		\item the unique function $f\colon V(F) \rightarrow V(F')$ given by $f(\mathbf{x}) = f(\mathbf{y}_i)$ and $f(u_j) < f(u_{j+1})$ for all $1 \le j \le v(F) - r - 1$ is an isomorphism between $F$ and $F'$.
	\end{enumerate}
	Moreover, let $\mathcal{F} = \bigcup_{i \in [s]} \mathcal{F}_i$. Note that for $\mathbf{x} = \emptyset$ we have $\mathcal{F}_i = \mathcal{F}_j$, thus each $F' \in \mathcal{F}$ appears $t$ times in $\mathcal{F}$. This, however, plays no role in the proof.
					
	We aim to apply Janson's inequality to deduce that the actual number of (hyper)graphs $F' \in \mathcal{F}$ that appear in $G$ is zero with sufficiently low probability. Recall the parameters associated with Janson's inequality,
	\begin{align*}
		X      & = \sum_{F' \in \mathcal{F}} X_{F'},                                                                                   \\
		\mu    & = \mathbb{E}[X] = \sum_{F' \in \mathcal{\mathcal{F}}} p^{e(F)},                                                       \\
		\delta & = \sum_{F' \sim F''}^{\mathcal{F}} \mathbb{E}[X_i X_j] = \sum_{F' \sim F''}^{\mathcal{F}} p^{2e(F) - e(F' \cap F'')}, 
	\end{align*}
	where $X_{F'}$ is an indicator random variable for $F' \subseteq G$ and $\sum^{\mathcal{F}}$ denotes that the sum runs over pairs of elements of $\mathcal{F}$. By  Janson's inequality we have
	$$
	\Pr[X = 0] \le \Pr[X \le \mu /2 ] \le e^{- \frac{\mu^2}{8 (\mu + \delta)}},
	$$
	thus it suffices to show 
	$$
	\mu \ge C' t \cdot v(F) \log n \qquad \text{ and }\qquad \delta \le \mu^2 / (C' t \cdot v(F) \log n)
	$$
	for some sufficiently large constant $C' = C'(K)$. 
					
	Note that every subset $Q \subseteq S'$ of size $v(F) - r$ uniquely determines a member of $\mathcal{F}_i$, for each $i \in [t]$. Thus $|\mathcal{F}| = t \binom{|S'|}{v(F) - r}$ and 
	\begin{multline*}
		\mu = |\mathcal{F}| p^{e(F)} = t \binom{|S'|}{v(F) - r} p^{e(F)} \stackrel{(*)}{\ge} t \left( \frac{|S|}{2 (v(F) - r)} \right)^{v(F) - r} p^{(v(F) - r) m(F, \mathbf{x})} \\ \ge t \left( \frac{|S|}{2 v(F)} p^{m(F, \mathbf{x})} \right)^{v(F) - r}  \ge C' t \cdot v(F) \log n,
	\end{multline*}
	as required. In $(*)$ we used $|S'| \ge |S|/2$ and $m(F, \mathbf{x}) \ge e(F) / (v(F) - r)$ (note that this is true even in $r = 0$). The last inequality follows from $v(F) - r \ge 2$ and
	$$
	p = \Omega \left( \frac{v(F)^2 \log n}{|S|} \right)^{1/m(F, \mathbf{x})}.
	$$

	Next, we estimate $\delta$ by splitting the sum into two, depending on the intersection of $V(F' \cap F'')$ and $Y:= \sum_{i \in [t]} \mathbf{y}_i$,
	$$
	\delta =\sum_{\substack{F' \sim F''\\V(F'\cap F'') \cap Y = \emptyset}}^{\mathcal{F}} p^{2 e(F) - e(F' \cap F'')} + \sum_{\substack{F' \sim F''\\V(F'\cap F'') \cap Y \neq \emptyset}}^{\mathcal{F}} p^{2 e(F) - e(F' \cap F'')}.
	$$
	We denote the first sum with $\delta_1$ and the second with $\delta_2$. We start by estimating $\delta_1$. Note that if $V(F' \cap F'') \cap Y = \emptyset$ then $\mathbf{x} \cap V(F' \cap F'') = \emptyset$ (here we slightly abuse the notation by identifying $F' \cap F''$ with the corresponding subgraph of $F$), thus
	$$
	\frac{e(F' \cap F'')}{v(F' \cap F'') - 1} \le m(F, \mathbf{x}).
	$$
	Therefore, we can upper bound $\delta_1$ as
	$$
	\delta_1 \le \sum_{\substack{F' \sim F''\\V(F'\cap F'') \cap Y = \emptyset}}^{\mathcal{F}} p^{2 e(F)} p^{-(v(F' \cap F'') - 1) m(F, \mathbf{x})}.	
	$$
	As we are concerned only with lexicographical copies, the vertex set of $F'$ and $F''$ uniquely determines the edge set of $F'$ and $F''$, respectively. Thus we iterate over copies $F', F''$ by first choosing $j \ge 2$ (the number of vertices in $V(F' \cap F'')$; note that if $F'$ and $F''$ have only one vertex in common then they do not share edges and, consequently, $F' \not \sim F''$), then picking $j$ vertices from $S'$, two sets of $v(F) - r - j$ vertices from the remaining vertices in $S'$ and two $r$-tuples from $\mathcal{Y}$,
	\begin{align*}
		\delta_1 & \le \sum_{j=2}^{v(F) - r} \binom{|S'|}{j}  \left(t \binom{|S'|-j}{v(F) - r - j}\right)^2 p^{2e(F)} p^{-(j-1)m(F, \mathbf{x})}                                    \\
		         & = \sum_{j=2}^{v(F) - r} \binom{|S'|}{j}  \left(t \binom{|S'|}{v(F) - r} \binom{ v(F) - r}{j} \binom{|S'|}{j}^{-1}\right)^2  p^{2e(F)} p^{-(j-1)m(F, \mathbf{x})} \\
		         & = \mu^2 \sum_{j=2}^{v(F) - r} \binom{|S'|}{j}^{-1}  \binom{v(F) - r}{j}^2   p^{-(j-1)m(F, \mathbf{x})}                                                           
	\end{align*}
	where in the last step we used $\mu = t \binom{|S'|}{v(F) - r}p^{e(F)}$. We further simplify the last expression by using simple algebraic manipulation and standard estimates for binomial coefficients,
	\begin{multline*}
		\delta_1 \le \mu^2 \sum_{j=2}^{v(F) - r} \left (\frac{e^2 v(F)^2}{|S'|} \right)^j p^{-(j-1)m(F, \mathbf{x})}
		\le \mu^2 \sum_{j=2}^{v(F) - r} \left( \frac{2 e^2 v(F)^2}{|S|} p^{-m(F, \mathbf{x})} \right)^j p^{m(F, \mathbf{x})}  \\
		\le \mu^2 \cdot 2 \left( \frac{2e^2 v(F)^2}{|S|} p^{-m(F, \mathbf{x})} \right)^2 p^{m(F, \mathbf{x})}  \le \mu^2 \cdot \frac{8e^4 v(F)^4}{|S|^2} p^{-m(F, \mathbf{x})},
	\end{multline*}
	where in the penultimate inequality we used $\sum_{j \ge 2} a^j < 2a^2$ for $a < 1/2$ (which holds for our choice of $p$). Finally, from the assumption $|S| \ge t (v(F) - r) \ge t v(F) / (r + 1)$ we conclude
	$$
	\delta_1 \le \frac{\mu^2}{t \cdot v(F)} \cdot (r+1) \frac{8e^4 v(F)^4}{|S|} p^{-m(F, \mathbf{x})} \le \frac{\mu^2}{C' t \cdot v(F) \log n},
	$$
	which follows from
	$$
	p = \Omega_{r} \left( \frac{v(F)^4 \log n}{|S|} \right)^{1/m(F, \mathbf{x})}.
	$$
					
	Next, we analyse the case $V(F' \cap F'') \cap Y \neq \emptyset$. If $\mathbf{x} = \emptyset$ then clearly there are no such pairs as $Y = \emptyset$, and consequently $\delta_2 = 0$. Thus for the rest of the proof we assume $r > 0$. Note that then we necessarily have $F', F'' \in \mathcal{F}_i$ as the $r$-tuples in $\mathcal{Y}$ are disjoint, thus
	$$
	\delta_2 = \sum_{i \in [t]} \sum_{\substack{F' \sim F''\\V(F' \cap F'') \cap Y \neq \emptyset}}^{\mathcal{F}_i} p^{2e(F) - e(F' \cap F'')}.
	$$ 
	Consider some $F', F'' \in \mathcal{F}_i$ in the above sum and let $H = F' \cap F''$. Observe that $\mathbf{y}_i \subseteq H$, hence $j = |V(H) \cap S'| = v(H) - r > 0$. By the definition we have
	$$
	e(H) / (v(H) - r) \le m(F, \mathbf{x}),
	$$
	and therefore
	$$
	\delta_2 \le \sum_{i \in [t]} \sum_{\substack{F' \sim F''\\V(F' \cap F'') \cap Y \neq \emptyset}}^{\mathcal{F}_i} p^{2e(F)} p^{-j \cdot m(F, \mathbf{x})}.
	$$
	Furthermore, note that $j > 0$ as otherwise $e(H) = 0$ (follows from the assumption that $\mathbf{y}_i$ is an independent set in $F'$ and $F''$), which contradicts $F' \sim F''$. For each $i \in [t]$ we calculate the sum in the upper bound on $\delta_2$ associated with $i$ as follows: we first choose $j \ge 1$ (the number of vertices in $V(F' \cap F'') \cap S'$), then pick $j$ vertices from $S'$ and two sets of $v(F) - r - j$ vertices from the remaining vertices in $S'$. Note that each such choice gives a unique pair $F' \sim F''$ in $\mathcal{F}_i$. Therefore, we have
	\begin{align}
		\delta_2 & \le \sum_{i \in [t]} \sum_{j = 1}^{v(F) - r} \binom{|S'|}{j} \binom{|S'| - j}{v(F) - r - j}^2 p^{2e(F)} p^{-j \cdot m(F, \mathbf{x})} \nonumber \\
		         & \le t \sum_{j = 1}^{v(F) - r} \binom{|S'|}{j} \binom{|S'| - j}{v(F) - r - j}^2 p^{2e(F)} p^{-j \cdot m(F, \mathbf{x})}. \label{eq:plug_back}    
	\end{align}
	We simplify binomial coefficients similarly as in the calculation of $\delta_1$,
	\begin{align*}
		\binom{|S'|}{j} \binom{|S'| - j}{v(F) - r - j}^2 & = \binom{|S'|}{j} \left( \binom{|S'|}{v(F) - r} \binom{v(F) - r}{j} \binom{|S'|}{j}^{-1} \right)^2 \\
		                                                 & \le \binom{|S'|}{v(F) - r}^2 \left ( \frac{e^2 v(F)^2}{|S'|}\right )^{j}.                          
	\end{align*}
	Finally, by plugging the previous estimate back into \eqref{eq:plug_back} we get 
	\begin{multline*}
		\delta_2 \le t \sum_{j = 1}^{v(F) - r} \binom{|S'|}{v(F) - r}^2 \left ( \frac{e^2 v(F)^2}{|S'|}\right )^{j} p^{2e(F)} p^{-j \cdot m(F, \mathbf{x})} 
		\\
		\le \frac{\mu^2}{t} \sum_{j = 1}^{v(F) - r} \left( \frac{2e^2 v(F)^2}{|S|} p^{-m(F, \mathbf{x})} \right)^{j} 
		\le \frac{\mu^2}{t} \cdot 2 \left( \frac{2e^2 v(F)^2}{|S|} p^{-m(F, \mathbf{x})} \right)
		\le \frac{\mu^2}{C' t \cdot v(F) \log n}. 
	\end{multline*}
	In the penultimate inequality we used $\sum_{j \ge 1} a^j < 2a$ for $a < 1/2$ (which holds for our choice of $p$), and the last inequality follows from
	$$
	p = \Omega \left( \frac{v(F)^3 \log n}{|S|} \right)^{1/m(F, \mathbf{x})}.
	$$
\end{proof}

\section{Absorbers} \label{sec:absorbers}

In this section we formally define the notion of absorbers used in the proof of our two main theorems.

\begin{definition}[Absorber]
Let $k$ be an integer, let $A$ be a graph (hypergraph) and $\mathbf{a}, \mathbf{b} \in V(A)^k$ disjoint $k$-tuples of vertices of $A$. Given a subset $X \subseteq V(A)$, we say that $A$ is an \emph{$(\mathbf{a}, \mathbf{b}, X)$-absorber} if for every subset $X' \subseteq X$ there exists 
a $k$-path ($k$-tight-path) $P \subseteq A$ from $\mathbf{a}$ to $\mathbf{b}$ such that $V(P) = V(A) \setminus X'$.
\end{definition}

As mentioned in Section \ref{sec:overview_absorber}, our goal is to find an absorber in $\Gnp$ and $\Gknp$ for a large subset $X$. This is accomplished by the following lemma.

\begin{lemma} \label{lemma:absorbing_path}
  Given an integer $k \ge 2$, there exists $C > 0$ such that if
  \begin{enumerate}[(i)]
    \item  $G \sim \Gnp$ for $p^k \ge C \log^8 n/n$ or 
    \item  $G \sim \Gknpp$ for $p \ge C \log^8 n/ n$,
  \end{enumerate}  
  then $G$ a.a.s contains an $(\mathbf{a}, \mathbf{b}, X)$-absorber $A$ with at most $v(A) \le n/2$ vertices, where $X \subset V(G)$ is a subset of size $|X| = \lfloor n / 16\log^2 n \rfloor$ and $\mathbf{a}, \mathbf{b} \in (V(G) \setminus X)^k$ are disjoint $k$-tuples.
\end{lemma}

In the rest of this section we prove Lemma \ref{lemma:absorbing_path}.

\subsection{Proof of Lemma \ref{lemma:absorbing_path}}

Let $X = \{x_1, \ldots, x_m\} \subseteq V(G)$ be a subset of vertices. Our strategy for constructing an absorber for $X$ consists of two steps. In the first step we find an $(\mathbf{a}_i, \mathbf{b}_i, \{x_i\})$-absorber $A_i$ (a \emph{single vertex} absorber) for each $x_i \in X$, such that they are pairwise disjoint. In the second step we find a $k$-connecting-path from $\mathbf{b}_i$ to $\mathbf{a}_{i+1}$, for every $1 \le i \le m - 1$, such that they are pairwise disjoint and also disjoint from $A_i$'s. It is easy to see that this gives an $(\mathbf{a}_1, \mathbf{b}_m, X)$-absorber: given $X' \subseteq X$, for every $x_i \in X$ we choose a $k$-path in $A_i$ depending on whether $x_i \in X'$ or not. We apply the similar strategy in the hypergraph case, with $k$-tight-paths replacing both $k$-paths and $k$-connecting-paths.

Having Corollaries \ref{cor:connecting} and \ref{cor:connecting_hyper} at hand, once we find $(\mathbf{a}_i, \mathbf{b}_i, \{x_i\})$-absorbers we can easily `connect' them into a $(\mathbf{a}_1, \mathbf{b}_m, X)$-absorber. Therefore, the real challenge is to find many single-vertex absorbers. We now define the main building block for constructing such absorbers, the family of graphs $\Jkl$ and hypergraphs $\HJkl$. Let 
$$
	\Wkl = \{x\} \cup \bigcup_{i \in [\ell]} \{w_{i,1}, \ldots, w_{i,2k}\}
$$
and set 
\begin{align*}
	\mathbf{w}_i^a &= (w_{i, 1}, \ldots, w_{i, k}), \\
	\mathbf{w}_i^b &= (w_{i, k + 1}, \ldots, w_{i, 2k})
\end{align*}
for every $i \in [\ell]$.

\begin{itemize}
	\item `Backbone' graph $\Jkl$ is a graph on the vertex set $W_\ell^k$ and the edge set given by the union of following graphs (see Figure \ref{fig:jumping_jack}): 
	\begin{itemize}
		\item the $(k, 2k + 1)$-path $(\mathbf{w}_1^a, x, \mathbf{w}_1^b)$;
		\item the $(k, 2k)$-path $(\mathbf{w}_i^a, \mathbf{w}_i^b)$ for every $2 \le i \le \ell$;		
		\item the $(k, 2k)$-path $(\mathbf{w}_2^a, \overline{\mathbf{w}}_1^a)$;
		\item the $(k, 2k)$-path $(\mathbf{w}_{i+2}^a, \mathbf{w}_{i}^b)$ for every $1 \le i \le \ell - 2$; 
		\item the $(k, 2k)$-path $(\overline{\mathbf{w}}_{\ell}^b, \mathbf{w}_{\ell-1}^b)$,
	\end{itemize}

	\begin{figure}[!ht] 
		\label{fig:jumping_jack}		
		\centering
		\begin{tikzpicture}[scale = 0.6]
			\tikzstyle{blob} = [fill=black,circle,inner sep=1.5pt,minimum size=0.5pt]
			
			\foreach \a in {2,4}{
				\foreach \i in {1,...,4}{				
					\node[blob] (\a\i) at ({((\a * 3 + \i - 1)*360/30 - 288}:6) {};
					\node at ({(\a * 3 + \i - 1)*360/30 - 288}:7) {$w_{\a\i}$};				
				}

				\draw (\a1) to (\a2);
				\draw[bend right] (\a1) to (\a3);			
				\draw (\a2) to (\a3);
				\draw[bend right] (\a2) to (\a4);
				\draw (\a3) to (\a4);
			}

			\foreach \i in {1,...,4}
			{	
				\node[blob] (5\i) at ({(4 * 6  - \i - 2)*360/30 - 288}:6) {};			
				\node at ({(4 * 6  - \i - 2)*360/30 - 288}:7) {$w_{5\i}$};	
			}
			\draw (51) to (52);
			\draw[bend left] (51) to (53);			
			\draw (52) to (53);
			\draw[bend left] (52) to (54);
			\draw (53) to (54);

			\foreach \i in {1,...,4}
			{	
				\node[blob] (3\i) at ({5 * 6  - \i - 2)*360/30 - 288}:6) {};			
				\node at  ({(5 * 6  - \i - 2)*360/30 - 288}:7) {$w_{3\i}$};			
			}		
			\draw (31) to (32);
			\draw[bend left] (31) to (33);			
			\draw (32) to (33);
			\draw[bend left] (32) to (34);
			\draw (33) to (34);

			\node[blob] (04) at ({0 - 295}:6) {}; \node at ({0 - 297}:6.5) {$w_{14}$};
			\node[blob] (03) at ({1 * 360/30 - 295}:6) {}; \node at ({1 * 360/30 - 297}:6.5) {$w_{13}$};
			\node[blob] (02) at ({2 * 360/30 - 281}:6) {}; \node at ({2 * 360/30 - 279}:6.5) {$w_{12}$};
			\node[blob] (01) at ({3 * 360/30 - 281}:6) {}; \node at ({3 * 360/30 - 278}:6.6) {$w_{11}$};

			\node[blob] (x) at ({360/20 - 288}:8) {}; \node at ({360/20 - 288}:8.5) {$x$};

			\draw (01) to (02); \draw (02) to (x); \draw (x) to (03); \draw (03) to (04); \draw (02) to (03);
			\draw[bend left] (01) to (x); \draw[bend left] (x) to (04);

			\draw[bend left] (02) to (22); \draw[bend left] (01) to (22); \draw[bend left] (02) to (21);
			\draw[bend right] (04) to (32); \draw[bend right] (03) to (32); \draw[bend right] (03) to (31);
			\draw[bend left] (24) to (42); \draw[bend left] (23) to (42); \draw[bend left] (23) to (41);
			\draw[bend right] (34) to (52); \draw[bend right] (33) to (52); \draw[bend right] (33) to (51);

			\draw[bend left] (44) to (53); \draw[bend left] (43) to (53); \draw[bend left] (43) to (54);
		\end{tikzpicture}
		\caption{The graph $\mathrm{B}^2_5$}
	\end{figure}

	\item `Backbone' $(k+1)$-graph $\HJkl$ is a $(k+1)$-graph on the vertex set $\Wkl$ and the edge set given by the union of following edge-disjoint $(k+1)$-graphs:	
	\begin{itemize}
		\item the $(k, 2k + 1)$-tight-path $(\mathbf{w}_1^a, x, \mathbf{w}_1^b)$;
		\item the $(k, 2k)$-tight-path $(\mathbf{w}_i^a, \mathbf{w}_i^b)$ for every $2 \le i \le \ell$;		
		\item the $(k, 2k)$-tight-path $(\mathbf{w}_2^a, \overline{\mathbf{w}}_1^a)$;
		\item the $(k, 2k)$-tight-path $(\mathbf{w}_{i+2}^a, \mathbf{w}_{i}^b)$ for every $1 \le i \le \ell - 2$;
		\item the $(k, 2k)$-tight-path $(\overline{\mathbf{w}}_{\ell}^b, \mathbf{w}_{\ell - 1}^b)$.
	\end{itemize}
\end{itemize}


\begin{claim} \label{claim:absorber}
Let $A_x$ be a graph ($(k+1)$-graph) obtained from $\Jkl$ ($\HJkl$), for some $k \ge 2$ and $\ell \ge 4$, by adding an arbitrary $k$-connecting-path ($k$-tight-path) $U_i$ from $\mathbf{w}_i^b$ to $\mathbf{w}_{i+1}^a$ for every $1 \le i < \ell$, such that all these paths are pairwise vertex-disjoint and also disjoint from $\Jkl$ ($\HJkl$) (except for the $k$-tuples of vertices they connect). Then $A_x$ is an $(\mathbf{w}_1^a, \mathbf{w}_\ell^b, \{x\})$-absorber.
\end{claim}
\begin{proof}
Note that there are only two cases we need to consider,  $X' = \emptyset$ and $X' = \{x\}$. We specify a desired path from $\mathbf{w}_1^a$ to $\mathbf{w}_\ell^b$ in each case by giving the  ordering in which we traverse the vertices of such a path (see Figure \ref{fig:absorb_example}):
\begin{multicols}{2}
	 $(i)$ $X' = \emptyset$
		\begin{align*}
			\mathbf{w}_1^a, x, &\mathbf{w}_1^b, U_1, \mathbf{w}_2^a,\\ 
			&\mathbf{w}_2^b, U_2, \mathbf{w}_3^a, \\
			&\qquad \vdots \\
			& \mathbf{w}_{\ell - 1}^b, U_{\ell-1}, \mathbf{w}_{\ell}^a, \mathbf{w}_{\ell}^b.
		\end{align*}

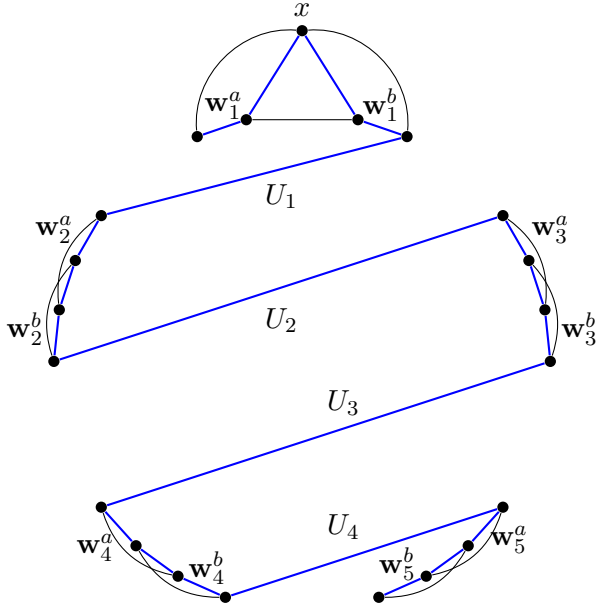
	\columnbreak
	$(ii)$ $X' = \{x\}$
	\begin{align*}
		\mathbf{w}_1^a, & \overline{\mathbf{w}}_2^a, \overline{U}_1, \overline{\mathbf{w}}_1^b, \\
		&\overline{\mathbf{w}}_3^a, \overline{U}_2, \overline{\mathbf{w}}_2^b, \\
		&\overline{\mathbf{w}}_4^a, \overline{U}_3, \overline{\mathbf{w}}_3^b, \\
		 &\qquad \vdots \\
		 &\overline{\mathbf{w}}_{\ell}^a, \overline{U}_{\ell - 1}, \overline{\mathbf{w}}_{\ell - 1}^b, \mathbf{w}_{\ell}^b.
	\end{align*}
\end{multicols}

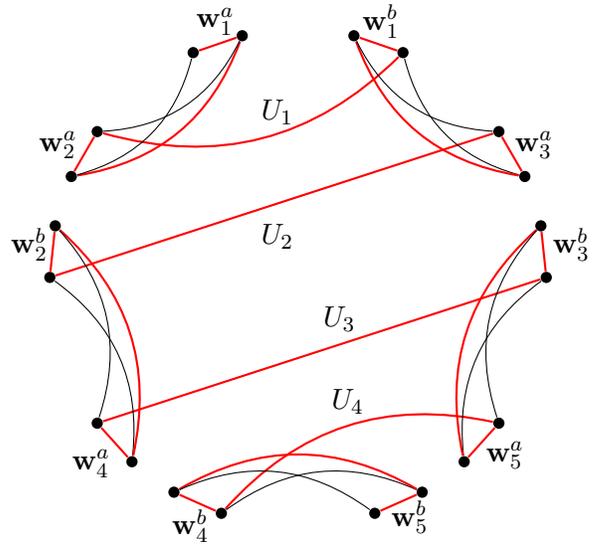
\begin{figure}[!ht] 
		\label{fig:absorb_example}		
		\centering
		\begin{tikzpicture}[scale = 0.55]
			\tikzstyle{blob} = [fill=black,circle,inner sep=1.5pt,minimum size=0.5pt]
			
			\foreach \a in {2,4}{
				\foreach \i in {1,...,4}{				
					\node[blob] (\a\i) at ({((\a * 3 + \i - 1)*360/30 - 288}:6) {};					
				}

				\draw[color=blue, thick] (\a1) to (\a2);
				\draw[bend right] (\a1) to (\a3);			
				\draw[color=blue, thick]  (\a2) to (\a3);
				\draw[bend right] (\a2) to (\a4);
				\draw[color=blue, thick]  (\a3) to (\a4);
			}

			\foreach \i in {1,...,4}
			{	
				\node[blob] (5\i) at ({(4 * 6  - \i - 2)*360/30 - 288}:6) {};							
			}
			\draw[color=blue, thick]  (51) to (52);
			\draw[bend left] (51) to (53);			
			\draw[color=blue, thick]  (52) to (53);
			\draw[bend left] (52) to (54);
			\draw[color=blue, thick]  (53) to (54);

			\foreach \i in {1,...,4}
			{	
				\node[blob] (3\i) at ({5 * 6  - \i - 2)*360/30 - 288}:6) {};							
			}		
			\draw[color=blue, thick]  (31) to (32);
			\draw[bend left] (31) to (33);			
			\draw[color=blue, thick]  (32) to (33);
			\draw[bend left] (32) to (34);
			\draw[color=blue, thick]  (33) to (34);

			\node[blob] (04) at ({0 - 295}:6) {}; 
			\node[blob] (03) at ({1 * 360/30 - 295}:6) {}; 
			\node[blob] (02) at ({2 * 360/30 - 281}:6) {}; 
			\node[blob] (01) at ({3 * 360/30 - 281}:6) {}; 

			\node[blob] (x) at ({360/20 - 288}:8) {}; \node at ({360/20 - 288}:8.5) {$x$};

			\draw[color=blue, thick]  (01) to (02); \draw[color=blue, thick]  (02) to (x); 
			\draw[color=blue, thick]  (x) to (03); 
			\draw[color=blue, thick]  (03) to (04); 
			\draw (02) to (03); \draw[bend left=50] (01) to (x); \draw[bend left=50] (x) to (04);

			\draw[color=blue, thick] (04) to (21); \node at (-0.5, 4) {$U_1$};
			\draw[color=blue, thick] (24) to (31); \node at (-0.5, 1) {$U_2$};
			\draw[color=blue, thick] (34) to (41); \node at (1, -1) {$U_3$};
			\draw[color=blue, thick] (44) to (51); \node at (1, -4) {$U_4$};

			\node at (-1.9, 6.2) {$\mathbf{w}_1^a$}; \node at (1.9, 6.2) {$\mathbf{w}_1^b$};
			\node at (-6, 3.2) {$\mathbf{w}_2^a$}; \node at (-6.7, 0.8) {$\mathbf{w}_2^b$};
			\node at (6, 3.2) {$\mathbf{w}_3^a$}; \node at (6.7, 0.8) {$\mathbf{w}_3^b$};
			\node at (-5, -4.5) {$\mathbf{w}_4^a$}; \node at (-2.3, -5) {$\mathbf{w}_4^b$};
			\node at (5, -4.3) {$\mathbf{w}_5^a$}; \node at (2.3, -4.9) {$\mathbf{w}_5^b$};

			\begin{scope}[shift={(16,0)}]
				\foreach \a in {2,4}{
					\foreach \i in {1,...,4}{				
						\node[blob] (\a\i) at ({((\a * 3 + \i - 1)*360/30 - 288}:6) {};					
					}

					\draw[color=red, thick] (\a1) to (\a2);
					\draw[color=red, thick] (\a3) to (\a4);
				}

				\foreach \i in {1,...,4}
				{	
					\node[blob] (5\i) at ({(4 * 6  - \i - 2)*360/30 - 288}:6) {};								
				}
				\draw[color=red, thick] (51) to (52);				
				\draw[color=red, thick] (53) to (54);

				\foreach \i in {1,...,4}
				{	
					\node[blob] (3\i) at ({5 * 6  - \i - 2)*360/30 - 288}:6) {};								
				}		
				\draw[color=red, thick] (31) to (32);				
				\draw[color=red, thick] (33) to (34);

				\node[blob] (04) at ({0 - 295}:6) {}; 
				\node[blob] (03) at ({1 * 360/30 - 295}:6) {}; 
				\node[blob] (02) at ({2 * 360/30 - 281}:6) {}; 
				\node[blob] (01) at ({3 * 360/30 - 281}:6) {}; 

				\draw[color=red, thick] (01) to (02); \draw[color=red, thick] (03) to (04); 

				\draw[color=red, thick, bend left] (04) to (21); \node at (-0.5, 4) {$U_1$};
				\draw[color=red, thick] (24) to (31); \node at (-0.5, 1) {$U_2$};
				\draw[color=red, thick] (34) to (41); \node at (1, -1) {$U_3$};
				\draw[color=red, thick, bend left] (44) to (51); \node at (1.2, -3) {$U_4$};

				\draw[bend left, color=red, thick] (02) to (22); \draw[bend left] (01) to (22); \draw[bend left] (02) to (21);
				\draw[bend right] (04) to (32); \draw[bend right, color=red, thick] (03) to (32); \draw[bend right] (03) to (31);
				\draw[bend left] (24) to (42); \draw[bend left, color=red, thick] (23) to (42); \draw[bend left] (23) to (41);
				\draw[bend right] (34) to (52); \draw[bend right, color=red, thick] (33) to (52); \draw[bend right] (33) to (51);

				\draw[bend left] (44) to (53); \draw[bend left, color=red, thick] (43) to (53); \draw[bend left] (43) to (54);

				\node at (-2, 6.2) {$\mathbf{w}_1^a$}; \node at (2, 6.2) {$\mathbf{w}_1^b$};
				\node at (-5.8, 3.2) {$\mathbf{w}_2^a$}; \node at (-6.5, 0.8) {$\mathbf{w}_2^b$};
				\node at (5.7, 3.2) {$\mathbf{w}_3^a$}; \node at (6.6, 0.8) {$\mathbf{w}_3^b$};
				\node at (-5, -4.5) {$\mathbf{w}_4^a$}; \node at (-2.6, -6) {$\mathbf{w}_4^b$};
				\node at (5, -4.3) {$\mathbf{w}_5^a$}; \node at (2.7, -5.8) {$\mathbf{w}_5^b$};
			\end{scope}
		\end{tikzpicture}
		\caption{Left: the $k$-path from $\mathbf{w}_1^a$ to $\mathbf{w}_{5}^b$ which includes $x$. Right: the $k$-path from $\mathbf{w}_1^a$ to $\mathbf{w}_{5}^b$ without $x$. Note that both $k$-paths use all other vertices of $A_x$.}
	\end{figure}

\end{proof}

With Claim \ref{claim:absorber} in mind, the first step in constructing many single-vertex absorbers is finding many disjoint copies of the `backbone' (hyper)graph. In order to apply Corollary \ref{cor:factor} to deduce the existence of such copies, we estimate the $m_1$-density of these graphs (see Section \ref{sec:corollaries}.

\begin{claim}
\label{claim:abs_dens}
Let $\ell \in \mathbb{N}$ be an odd integer such that $\ell \ge 3$. Then $m_1(\Jkl) \le k$ and $m_1(\HJkl) \le 1$.
\end{claim}
\begin{proof}
We only prove the part of the claim concerning graph $\Jkl$ as the other part follows an analogous argument. Given a (hyper)graph $F$ and an integer $k \in \mathbb{N}$, we say that $F$ is $k$-degenerate if there exists an ordering $V(H) = \{v_1, \ldots, v_h\}$ of its vertices such that for each $i \in [h]$ there are at most $k$ edges $e \in F$ such that $v_i \in e$ and $e \subseteq \{v_1, \ldots, v_i\}$. Moreover, any ordering which witnesses that $H$ is $k$-degenerate is called \emph{$k$-degenerate ordering}. It is a standard exercise to show that (i) a $k$-degenerate (hyper)graph $F$ has at most $e(F) \le k (v(F) - 1)$ edges and (ii) every subgraph $F' \subseteq F$ is also $k$-degenerate. Having these two facts in mind, in order to show $m_1(\Jkl) \le k$ it suffices to show that $\Jkl$ is $k$-degenerate: for every subgraph $F' \subseteq \Jkl$ we have $v(F') \le k (v(F') - 1)$ as $F'$ is then also $k$-degenerate, thus
$$
	\frac{e(F')}{v(F') - 1} \le k.
$$

We verify that the following ordering of the vertices of $\Jkl$ witnesses its $k$-degeneracy (see Figure \ref{fig:jumping_jack_ordering}),
\begin{align*}
V(\Jkl) = (	x, &\overline{\mathbf{w}}_1^a, \\
			&\overline{\mathbf{w}}_2^a, \mathbf{w}_2^b, \overline{\mathbf{w}}_4^a, \mathbf{w}_4^b, \ldots, \overline{\mathbf{w}}_{\ell-3}^a, \mathbf{w}_{\ell-3}^b, \overline{\mathbf{w}}_{\ell-1}^a, \mathbf{w}_{\ell-1}^b, \\
 			&\mathbf{w}_\ell^b, \overline{\mathbf{w}}_{\ell}^a, \mathbf{w}_{\ell-2}^b, \overline{\mathbf{w}}_{\ell-2}^a, \mathbf{w}_{\ell-4}^b, \overline{\mathbf{w}}_{\ell-4}^a, \ldots, \mathbf{w}_{3}^b, \overline{\mathbf{w}}_{3}^a, \\
 			&\mathbf{w}_1^b).
\end{align*}

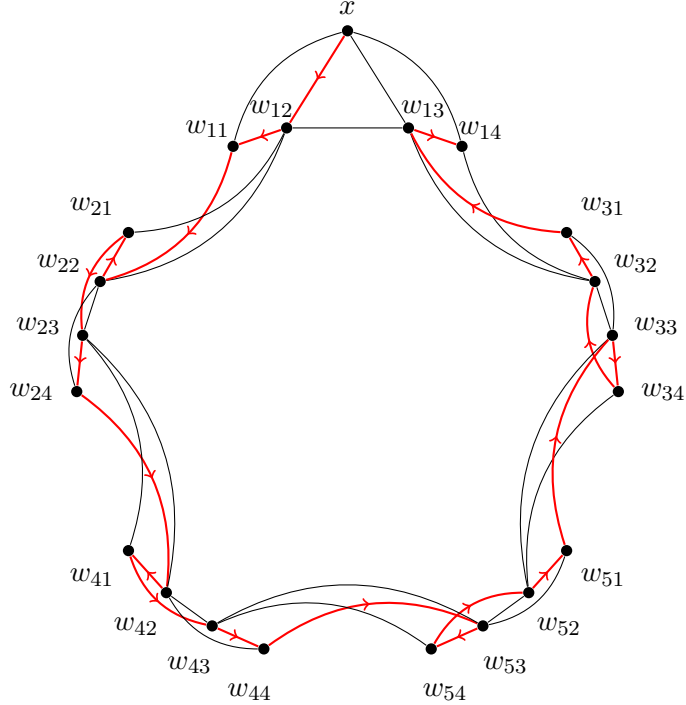
\begin{figure}[!ht] 
		\label{fig:jumping_jack_ordering}		
		\centering
		\begin{tikzpicture}[scale = 0.6, decoration={markings, mark=at position 0.5 with {\arrow{>}}}]

			\tikzstyle{blob} = [fill=black,circle,inner sep=1.5pt,minimum size=0.5pt]
			
			\foreach \a in {2,4}{
				\foreach \i in {1,...,4}{				
					\node[blob] (\a\i) at ({((\a * 3 + \i - 1)*360/30 - 288}:6) {};
					\node at ({(\a * 3 + \i - 1)*360/30 - 288}:7) {$w_{\a\i}$};				
				}

				\draw[red, thick, postaction={decorate}] (\a2) to (\a1);
				\draw[bend right, red, thick, postaction={decorate}] (\a1) to (\a3);			
				\draw (\a2) to (\a3);
				\draw[bend right] (\a2) to (\a4);
				\draw[red, thick, postaction={decorate}] (\a3) to (\a4);
			}

			\foreach \i in {1,...,4}
			{	
				\node[blob] (5\i) at ({(4 * 6  - \i - 2)*360/30 - 288}:6) {};			
				\node at ({(4 * 6  - \i - 2)*360/30 - 288}:7) {$w_{5\i}$};	
			}
			\draw[red, thick, postaction={decorate}] (52) to (51);
			\draw[bend left] (51) to (53);			
			\draw (52) to (53);
			\draw[bend left, red, thick, postaction={decorate}] (54) to (52);
			\draw[red, thick, postaction={decorate}] (53) to (54);

			\foreach \i in {1,...,4}
			{	
				\node[blob] (3\i) at ({5 * 6  - \i - 2)*360/30 - 288}:6) {};			
				\node at  ({(5 * 6  - \i - 2)*360/30 - 288}:7) {$w_{3\i}$};			
			}		
			\draw[red, thick, postaction={decorate}] (32) to (31);
			\draw[bend left] (31) to (33);			
			\draw (32) to (33);
			\draw[bend left, red, thick, postaction={decorate}] (34) to (32);
			\draw[red, thick, postaction={decorate}] (33) to (34);

			\node[blob] (04) at ({0 - 295}:6) {}; \node at ({0 - 297}:6.5) {$w_{14}$};
			\node[blob] (03) at ({1 * 360/30 - 295}:6) {}; \node at ({1 * 360/30 - 297}:6.5) {$w_{13}$};
			\node[blob] (02) at ({2 * 360/30 - 281}:6) {}; \node at ({2 * 360/30 - 279}:6.5) {$w_{12}$};
			\node[blob] (01) at ({3 * 360/30 - 281}:6) {}; \node at ({3 * 360/30 - 278}:6.6) {$w_{11}$};

			\node[blob] (x) at ({360/20 - 288}:8) {}; \node at ({360/20 - 288}:8.5) {$x$};

			\draw[red, thick, postaction={decorate}] (02) to (01); \draw[red, thick, postaction={decorate}] (x) to (02); \draw (x) to (03); \draw[red, thick, postaction={decorate}] (03) to (04); \draw (02) to (03);
			\draw[bend left] (01) to (x); \draw[bend left] (x) to (04);

			\draw[bend left] (02) to (22); \draw[bend left, red, thick, postaction={decorate}] (01) to (22); \draw[bend left] (02) to (21);
			\draw[bend right] (04) to (32); \draw[bend right] (03) to (32); \draw[bend left, red, thick, postaction={decorate}] (31) to (03);
			\draw[bend left, red, thick, postaction={decorate}] (24) to (42); \draw[bend left] (23) to (42); \draw[bend left] (23) to (41);
			\draw[bend right] (34) to (52); \draw[bend right] (33) to (52); \draw[bend left, red, thick, postaction={decorate}] (51) to (33);

			\draw[bend left, red, thick, postaction={decorate}] (44) to (53); \draw[bend left] (43) to (53); \draw[bend left] (43) to (54);
		\end{tikzpicture}
		\caption{The ordering of $V(\mathrm{B}_\ell^k)$. An arrow pointing from $v$ to $w$ means that $v$ comes before $w$.}
	\end{figure}

For each vertex $v \in V(H)$ let $e(v)$ denote the number of edges which contain $v$ and do not contain any vertices succeeding $v$ in the described ordering. We need to check $e(v) \le k$ for every $v \in V(H)$. We distinguish different cases depending on the position of $v$:
\begin{itemize}
\item $v \in \{x\} \cup \overline{\mathbf{w}}_1^a$: there are at most $k$ vertices which precede such $v$, thus $e(v) \le k$.

\item $v \in \overline{\mathbf{w}}_{i}^a$ for even $i \in \{2, 4, \ldots, \ell - 1\}$:
From the definition of $\Jkl$ one can see that the only edges incident to $v$ belong to $k$-paths $\mathbf{w}_{i}^a \mathbf{w}_{i}^b$ and $\mathbf{w}_{i}^a \mathbf{w}_{i -2}^b$ (if $i \ge 4$), that is, $\mathbf{w}_{i}^a \mathbf{w}_{i}^b$ and $\mathbf{w}_{i}^a \overline{\mathbf{w}}_{1}^a$ (if $i = 2$). Since the vertices in $\mathbf{w}_{i}^b$ appear after $v$, the edges from the first $k$-path do not contribute to $e_{v}$. Moreover, the vertices of the second set of $k$-paths are ordered as $(\mathbf{w}_{i-2}^b, \overline{\mathbf{w}}_{i}^a)$ and $(\overline{\mathbf{w}}_{1}^a, \overline{\mathbf{w}}_{2}^a)$, respectively, which is easily seen to be $k$-degenerate. 
Therefore, we conclude $e(v) \le k$.

 \item $v \in \mathbf{w}_{i}^b$ for even $i \in \{2, 4, \ldots, \ell - 1\}$:  
 The edges incident to $v$ are obtained from the union of $k$-paths $\mathbf{w}_{i}^a \mathbf{w}_{i}^b$ and 
 $\mathbf{w}_{i+2}^a \mathbf{w}_{i}^b$ (if $i < \ell - 1$), that is, $\mathbf{w}_{i}^a \mathbf{w}_{i}^b$ and 
 $\overline{\mathbf{w}}_{\ell}^b \mathbf{w}_{i}^b$ (if $i = \ell - 1$). As the vertices of $\mathbf{w}_{i+2}^a$ and $\overline{\mathbf{w}}_{\ell}^b$ appear after $v$, respectively, they do not contribute to $e(v)$. The first set of $k$-paths are ordered as $(\overline{\mathbf{w}}_{i}^a \mathbf{w}_{i}^b)$, thus we conclude $e(v) \le k$.

 \item $v \in \mathbf{w}_{i}^b$ for odd $i \in \{3, 5, \ldots, \ell\}$: 
 The edges incident to $v$ belong to the union of $k$-paths $\mathbf{w}_{i}^a\mathbf{w}_{i}^b$ and $\mathbf{w}_{i+2}^a \mathbf{w}_{i}^b$ (if $i < \ell$), that is, $\mathbf{w}_{i}^a\mathbf{w}_{i}^b$ and $\overline{\mathbf{w}}_{i}^b \mathbf{w}_{\ell - 1}^b$ (if $i = \ell$). 
 Since the vertices from $\mathbf{w}_{i}^a$ appear after $v$, edges incident to them do not contribute to $e(v)$. The vertices of the second set of $k$-paths 
 appear in the order $(\overline{\mathbf{w}}_{i+2}^a,\mathbf{w}_{i}^b)$ (if $i < \ell$) and $(\mathbf{w}_{\ell-1}^b, \mathbf{w}_i^b)$ (if $i = \ell$), from which we conclude $e(v) \le k$.
 
 \item $v \in \overline{\mathbf{w}}_{i}^a$ for odd $i \in \{3, 5, \ldots, \ell\}$: The edges incident to $v$ are obtained from the union of $k$-paths $\mathbf{w}_{i}^a\mathbf{w}_{i}^b$ and $\mathbf{w}_{i}^a \mathbf{w}_{i-2}^b$. Since the vertices of $\mathbf{w}_{i-2}^b$ appear after $v$, edges incident to them do not contribute to $e(v)$. The vertices of the first set if $k$-paths appear in the order $(\mathbf{w}_{i}^b, \overline{\mathbf{w}}_{i}^a)$ which implies $e(v) \le k$.


 \item Let $v \in \mathbf{w}_1^b$: The edges incident to $v$ are obtained from the union of $2k$-paths 
 $\mathbf{w}_1^a x \mathbf{w}_1^b$ and 
 $\mathbf{w}_3^a \mathbf{w}_1^b$. The vertices of $\mathbf{w}_3^a$ appear after $v$ thus the edges from the second $k$-path do not contribute to $e(v)$.  
 The vertices of the first $k$-path appear in the order 
 $(x,\overline{\mathbf{w}}_1^a,\mathbf{w}_1^b)$ which is a $k$-degenerate ordering and thus $e(v) \le k$.

\end{itemize}
\end{proof}



Finally, we are ready to prove the main lemma of this section.

\begin{proof}[Proof of Lemma \ref{lemma:absorbing_path}]
Consider some $k \ge 2$. Let $[n] = W_1 \cup W_2 \cup W_3$ be an equipartition of $[n]$ and set $\ell = \log n$. From estimates on $m_1(\cdot)$ given by Claim \ref{claim:abs_dens}, we conclude that $G \sim \Gnp$ ($G \sim \Gknpp$) a.a.s satisfies the property of Corollary \ref{cor:factor} for $W = W_1$ and $F = \Jkl$ ($F = \HJkl$) for $p$ as stated. Let us denote this property by (PF). Furthermore, for $p = \Omega(\log^6 n / n)^{1/k}$ we have that $G \sim \Gnp$ a.a.s satisfies the property of Corollary \ref{cor:connecting} (Graph Connecting Lemma) for $W = W_2$ and $W = W_3$. Similarly, for $p = \Omega(\log^6 n / n)$ the random $(k+1)$-graph $G \sim \Gknpp$ a.a.s satisfies the property of Corollary \ref{cor:connecting_hyper} (Hypergraph Connecting Lemma) for $W = W_2$ and $W = W_3$. We denote this property by (PCONN). In the rest of the proof we show that these two properties suffice for the existence of a desired absorber.


By the property (PF), $G$ contains a family $\{F_i\}_{i \in [t]}$ of $t = n / 16 \log^2 n$ vertex-disjoint copies of $F = \Jkl$ ($F = \HJkl$) with all vertices being in $W_1$. Let $g_i \colon F \rightarrow F_i$ denote an isomorphism of $F$ into $F_i$, for each $i \in [t]$. Next, consider the family of $2k$-tuples $\{ (g_i(\mathbf{w}_j^b), g_i(\mathbf{w}_{j+1}^a) \}_{i \in [t], j \in [\ell - 1]}$. By the property (PCONN) there exists a family $\{U^i_{j} \subseteq G\}_{i \in [t], j \in [\ell - 1]}$ of vertex-disjoint $k$-connecting-paths ($k$-tight-paths), where $U_j^i$ is a $k$-connecting-path ($k$-tight-path) from $g_i(\mathbf{w}_j^b)$ to $g_i(\mathbf{w}^a_{j+1})$, with all internal vertices being in $W_2$. Claim \ref{claim:absorber} implies that the (hyper)graph $A_i$ given by $g_i(F) \cup \bigcup_{j \in [\ell - 1]} U_j^i$ (see Figure \ref{fig:step1}) is an $(\mathbf{a}_i, \mathbf{b}_i, x_i)$-absorber, where 
$$
	\mathbf{a}_i = g_i(\mathbf{w}_1^a), \quad \mathbf{b}_i = g_i(\mathbf{w}_{\ell}^b) \quad \text{ and } \quad x_i = g_i(x).
$$

\begin{figure}[!ht]
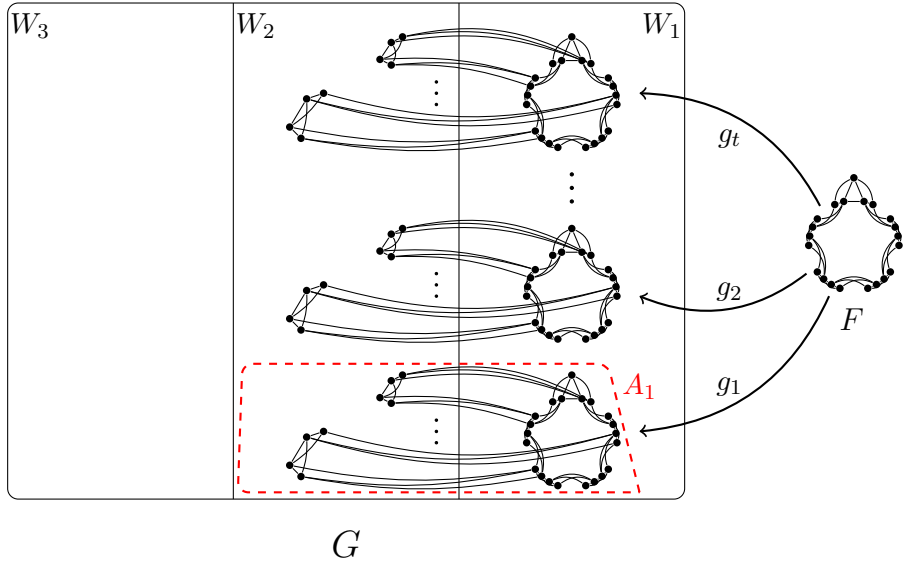
 
	\label{fig:step1}	

	\centering
	\begin{tikzpicture}[scale = 0.6]
		\tikzstyle{blob} = [fill=black,circle,inner sep=1pt,minimum size=0.1pt]		
		
		\draw[rounded corners] (0, 0) rectangle (15, 11); \node at (7.5, -1) {\Large $G$};
		\draw (10, 0) -- (10, 11);
		\draw (5, 0) -- (5, 11);

		\node at (14.5, 10.5) {$W_1$};
		\node at (5.5, 10.5) {$W_2$};
		\node at (0.5, 10.5) {$W_3$};

		\begin{scope}[scale = 0.25, shift = {(75, 22.5)}]
			\foreach \a in {2,4}{
				\foreach \i in {1,...,4}{				
					\node[blob] (\a\i) at ({((\a * 3 + \i - 1)*360/30 - 288}:4) {};		
				}

				\draw (\a1) to (\a2);
				\draw[bend right] (\a1) to (\a3);			
				\draw (\a2) to (\a3);
				\draw[bend right] (\a2) to (\a4);
				\draw (\a3) to (\a4);
			}

			\foreach \i in {1,...,4}
			{	
				\node[blob] (5\i) at ({(4 * 6  - \i - 2)*360/30 - 288}:4) {};				
			}
			\draw (51) to (52);
			\draw[bend left] (51) to (53);			
			\draw (52) to (53);
			\draw[bend left] (52) to (54);
			\draw (53) to (54);

			\foreach \i in {1,...,4}
			{	
				\node[blob] (3\i) at ({5 * 6  - \i - 2)*360/30 - 288}:4) {};				
			}		
			\draw (31) to (32);
			\draw[bend left] (31) to (33);			
			\draw (32) to (33);
			\draw[bend left] (32) to (34);
			\draw (33) to (34);

			\node[blob] (04) at ({0 - 295}:4) {}; 
			\node[blob] (03) at ({1 * 360/30 - 295}:4) {}; 
			\node[blob] (02) at ({2 * 360/30 - 281}:4) {}; 
			\node[blob] (01) at ({3 * 360/30 - 281}:4) {}; 

			\node[blob] (x) at ({360/20 - 288}:6) {}; 

			\draw (01) to (02); \draw (02) to (x); \draw (x) to (03); \draw (03) to (04); \draw (02) to (03);
			\draw[bend left] (01) to (x); \draw[bend left] (x) to (04);

			\draw[bend left] (02) to (22); \draw[bend left] (01) to (22); \draw[bend left] (02) to (21);
			\draw[bend right] (04) to (32); \draw[bend right] (03) to (32); \draw[bend right] (03) to (31);
			\draw[bend left] (24) to (42); \draw[bend left] (23) to (42); \draw[bend left] (23) to (41);
			\draw[bend right] (34) to (52); \draw[bend right] (33) to (52); \draw[bend right] (33) to (51);

			\draw[bend left] (44) to (53); \draw[bend left] (43) to (53); \draw[bend left] (43) to (54);
		\end{scope}	

		\node at (18.7, 4) {\large $F$};	

		\begin{scope}[scale = 0.25, shift = {(50, 5)}] \input{absorbah.pic} \end{scope}
		\begin{scope}[scale = 0.25, shift = {(50, 18)}] \input{absorbah.pic} \end{scope}
		\begin{scope}[scale = 0.25, shift = {(50, 35)}] \input{absorbah.pic} \end{scope}

		\draw[fill, black] (12.5, 7.2) circle [radius=1pt];
		\draw[fill, black] (12.5, 6.9) circle [radius=1pt];
		\draw[fill, black] (12.5, 6.6) circle [radius=1pt];

		\draw[bend right, ->, thick] (18, 6.5) to (14, 9); \node at (16, 8) {$g_t$};
		\draw[bend left, ->, thick] (17.7, 5) to (14, 4.5); \node at (16, 4.6) {$g_2$};
		\draw[bend left, ->, thick] (18.2, 4.5) to (14, 1.5); \node at (16, 2.5) {$g_1$};

		\draw[thick, rounded corners, dashed, red] (14, 0.15) -- (13.3, 3) -- (5.2, 3) -- (5.1, 0.15) -- (14, 0.15);
		\node[color=red] at (14, 2.5) {$A_1$};
	\end{tikzpicture}	

	\caption{Construction of disjoint single-vertex absorbers.}	
\end{figure}

Next, using vertices in $W_3$ we connect all $A_i$'s into a single absorber for the set $X = \{x_1, \ldots, x_t\}$. Consider the family of $2k$-tuples $\{(\mathbf{b}_i, \mathbf{a}_{i+1})\}_{i \in [t - 1]}$, and let $\{Q_i\}_{i \in [t - 1]}$ denote the family of $k$-connecting-paths ($k$-tight-paths) given by the property (PCONN), with all internal vertices being in $W_3$.  We claim that 
$$
	A = \bigcup_{i \in [t]} A_i \cup \bigcup_{i \in [t - 1]} Q_i
$$
is an $(\mathbf{a}_1, \mathbf{b}_t, X)$-absorber: Consider some subset $X' \subseteq X$ and for each $i \in [t]$ let $P_i \subseteq A_i$ be a $k$-path ($k$-tight-path) from $\mathbf{a}_i$ to $\mathbf{b}_i$ which contains $x_i$ iff $x_i \notin X'$ and, moreover, contains all other vertices in $A_i$. Such a path exists as $A_i$ is an $(\mathbf{a}_i, \mathbf{b}_i, x_i)$-absorber. Then
$$
	\mathbf{a}_1 \xrsquigarrow{$P_1$} \mathbf{b}_1 \xrsquigarrow{$Q_1$} \mathbf{a}_2 \xrsquigarrow{$P_2$} \mathbf{b}_2 \xrsquigarrow{$Q_2$} \; \ldots \; \xrsquigarrow{$Q_{t - 1}$} \mathbf{a}_t \xrsquigarrow{$P_t$} \mathbf{b}_t
$$
gives a $k$-path ($k$-tight-path) from $\mathbf{a}_1$ to $\mathbf{b}_t$ which contains all vertices in $A$ except those in $X'$. This concludes the proof of the lemma.
\end{proof}


\section{Proof of the main result}
\label{sec:main}

The last ingredient in the proof of our main results is the following classical result of Erd\H{o}s and R\'enyi \cite{erdos1964random}. Given $p = p(n) \in [0, 1]$, let $\mathcal{G}(n, n, p)$ denote a graph with the vertex set $A \cup B$, where $|A| = |B| = n$, formed by adding each possible edge between $A$ and $B$ with probability $p$, independently of all other edges.

\begin{theorem} \label{thm:matching}
    Let $p = p(n) \in [0, 1]$ be such that $np - \log n \rightarrow \infty$. Then $\mathcal{G}(n,n,p)$ contains a perfect matching with probability at least $1 - O(ne^{-np})$.
\end{theorem}

As the proof of both Theorem \ref{thm:main_graph} and Theorem \ref{thm:main_hypergraph} follows the same argument, we present it with the graph case in mind while pointing out the necessary changes for the hypergraph case, usually given in brackets.

\begin{proof}[Proof of Theorem \ref{thm:main_graph} and Theorem \ref{thm:main_hypergraph}]
    Suppose $p^k \ge C \log n^8 / n$ ($p \ge C \log^8 n / n$), for some sufficiently large $C > 0$, and let $q \in \mathbb R$ be such that $q^k \ge C' \log^8 n /n$ ($q \ge C' \log n / n$) and $p = 1 - (1 - q)^3$. We generate $G \sim \Gnp$ ($G \sim \Gknpp$) as the union $G_1 \cup G_2 \cup G_3$ where $G_1, G_2, G_3 \sim \mathcal{G}(n, q)$ ($G_1, G_2 \in \mathcal{G}^{(k+1)}(n, q)$). As each edge is present in at least one of the three graphs with probability exactly $1 - (1 - q)^3 = p$, we conclude $G_1 \cup G_2 \cup G_3 \sim \Gnp$. This is usually referred to as \emph{multiple exposure}.  
    
    First, from Lemma \ref{lemma:absorbing_path} we conclude that $G_1$ contains an $(\mathbf{a}, \mathbf{b}, X)$-absorber $A$ for some subset $X \subseteq [n]$ of size $|X| = n / 16 \log^2 n$ and disjoint $k$-tuples $\mathbf{a}, \mathbf{b} \in ([n] \setminus X)^k$. Next, we aim to cover all `unused' vertices $U := [n] \setminus V(A)$ with $O(n/\log^4 n)$ many vertex-disjoint $k$-paths ($k$-tight-paths). Let $t = \lceil \log^4 n \rceil$ and choose an arbitrary subset $U_X \subseteq X$ of size at most $t - 1$ such that $|U \cup U_X|$ is divisible by $t$. Consider an equitable partition of the set $U \cup U_X$ into $t$ parts $U_1, \ldots, U_t$. From $v(A) \le n/2$ we have that each $U_i$ is of size $|U_i| = s \in [n / (2 t), n/t]$. We iteratively construct a family of vertex-disjoint $k$-paths ($k$-tight-paths) $Q^j_1, \ldots, Q^j_s$ in $G_2$, for $1 \le j \le t$, such that $|V(Q^j_i) \cap U_{j'}| = 1$ for every $i \in [s]$ and $j' \in [j]$. For $j = t$, this gives a family of vertex-disjoint $k$-paths ($k$-tight-paths) $Q_1, \ldots, Q_s$ which cover all vertices in $U \cup U_X$. Moreover, in each step $j$ we only reveal edges of $G_2$ with one endpoint in $U_j$ and the other in $\bigcup_{j' < j} U_{j'}$. 
    
    For $j = 1$ this trivially holds since each vertex in $U_1$ forms a $k$-path of length 1. Note that in the hypergraph case same holds for every $j < k$, as every subset of $j < k$ vertices forms a $k$-tight-path. Let us assume that the assumption is true for some $j \ge 1$. To construct desired $k$-paths ($k$-tight-paths) for $j + 1$ we consider the auxiliary bipartite graph $B_{j+1}$ where one class corresponds to $U_{j+1}$ and the other class corresponds to $k$-paths ($k$-tight-paths) $\{Q_i^{j}\}_{i \in [s]}$ (that is, for each $Q_i^j$ we have one vertex which represents it), and
    $$
    E(B_{j+1}) = \{ \{Q_i^j, u\} \; : \; u \in U_{j+1} \text{ and } \{Q_i^j \cap U_{j'}, u\} \in G_2 \text{ for every } j' \in \{\max\{1, j - k + 1 \}, j\} \}.
    $$
    Recall that $Q_i^j \cap U_{j'}$ is a single vertex for every $j' \in [j]$. Moreover, as the edges edges of $G_2$ touching the set $U_{j+1}$ have not been exposed so far, we have $B_{j+1} \sim \mathcal{G}(s, s, q^{k'})$ where $k' = \min\{k, j\}$ (in the hypergraph case we have $B_{j+1} \sim \mathcal{G}(s, s, q)$). As $q^{k'} \ge C' \log^8 n / n \gg \log s / s$, Theorem \ref{thm:matching} tells us that with probability least $1 - 1/n$ (with room to spare) there exists a perfect matching $M_{j+1}$ in $B_{j+1}$. Consider one such perfect matching and let $u_i^{j+1} \in U_{i+1}$ be the vertex matched to the vertex corresponding to the path $Q_i^j$. From the definition of $B_{j+1}$ we conclude that extending $Q_i^j$ to $u_i^{j+1}$ for each $i \in [s]$ gives the desired family of $k$-paths ($k$-tight-paths) for $j + 1$. Finally, the probability that there exists a step $j \in [t]$ in which we were not able to find a perfect matching in $B_j$ is at most $t/n = o(1)$. As a perfect matching in each step can be found in polynomial time, this whole procedure requires polynomial time as well.
    
    Next, we merge $k$-paths $Q_1, \ldots, Q_s$ into one $k$-path ($k$-tight-path) from $\mathbf{a}$ to $\mathbf{b}$ by additionally using only vertices from $X \setminus U_X$. To this end, note that the set $W := X \setminus U_X$ is of size at least
    $$
    	|W| \ge \frac{n}{16 \log^2 n} - t \ge \frac{n}{17 \log^2 n}
    $$
    and, as $q^k \ge C' \log^8 n / n \ge C_{\ref{cor:connecting}} \log^6 n / |W|$ ($q \ge C_{\ref{cor:connecting_hyper}} \log^6 n / |W|$) we have that $G_3$ a.a.s satisfies the property of Corollary \ref{cor:connecting} (Corollary \ref{cor:connecting_hyper}) for $W$ (Graph and Hypergraph Connecting Lemma, respectively). For each $1 \le i \le s$, let $\mathbf{a}_i$ and $\mathbf{b}_i$ denote the $k$ left-most and right-most vertices of $Q_i$, that is,
    $$
    	\mathbf{a}_i = (Q_i \cap U_1, Q_i \cap U_2, \ldots, Q_i \cap U_k) \qquad \text{ and } \qquad \mathbf{b}_i = (Q_i \cap U_{\ell - k + 1}, \ldots, Q_i \cap U_{\ell}).
    $$
    (Recall that each $Q_i$ has exactly one vertex in each $U_j$). Consider the family of $2k$-tuples 
    $$
    	\{(\mathbf{b}, \mathbf{a}_1), (\mathbf{b}_1, \mathbf{a}_2), \ldots, (\mathbf{b}_{s - 1}, \mathbf{a}_s), (\mathbf{b}_s, \mathbf{a})\}
    $$ 
    and let $\{Z_i\}_{i \in [s+1]}$ be a family of vertex-disjoint $k$-connecting-paths ($k$-tight-paths) with all internal vertices being in $W = X \setminus U_X$, where
    \begin{itemize}
    	\item $Z_1$ is a $k$-connecting-path ($k$-tight-path) from $\mathbf{b}$ to $\mathbf{a}_1$, 
    	\item $Z_i$ is a $k$-connecting-path ($k$-tight-path) from $\mathbf{b}_{i-1}$ to $\mathbf{a}_i$, for $2 \le i \le s$, and
    	\item $Z_{s+1}$ is a $k$-connecting-path ($k$-tight-path) from $\mathbf{b}_{s}$ to $\mathbf{a}$.
	\end{itemize}
    Let $Z_X = \bigcup_{i \in [s+1]} V(Z_i) \cap W$ denotes the subset of vertices from $W$ which belong to these $k$-paths, and set $X' := U_X \cup Z_X$. Then
    $$
    	P' = \mathbf{b} \xrsquigarrow{$Z_1$} \mathbf{a}_1 \xrsquigarrow{$Q_1$} \mathbf{b}_1 
    	\xrsquigarrow{$Z_2$} \mathbf{a}_2 \xrsquigarrow{$Q_2$} \mathbf{b}_2 \xrsquigarrow{$Z_3$} \ldots \xrsquigarrow{$Z_{s}$} \mathbf{a_s} \xrsquigarrow{$Q_s$} \mathbf{b}_s 
    	\xrsquigarrow{$Z_{s+1}$} \mathbf{a}    	
    $$
    forms a $k$-path ($k$-tight-tight) from $\mathbf{b}$ to $\mathbf{a}$ which contains all the vertices in $G$ except $A \setminus (X' \cup \mathbf{a} \cup \mathbf{b})$. Finally, as $A$ is an $(\mathbf{a}, \mathbf{b}, X)$ absorber there exists a $k$-path ($k$-tight-path) $P_{X'} \subseteq A$ from $\mathbf{a}$ to $\mathbf{b}$ which contains all the vertices of $A$ except those in $X'$. Together with $P'$, such a path forms a Hamilton $k$-cycle ($k$-tight-cycle).
\end{proof}

\begin{remark}
    As we make use of a 3-round exposure of a random (hyper)graph, the algorithm obtained from the proof implicitly takes as an input three (hyper)graphs. However, as it is more conventional that the input graph is given at once, by deciding for each edge $e \in G$ and each $i \in \{1, 2, 3\}$ whether splitting its edges independently at random into three graphs we achieve the same result. 
\end{remark}

\section{Concluding remarks} \label{sec:remarks}

Our main contribution is an upper bound on the threshold for containing the square of a Hamilton cycle. In particular, we improve a result of K\"uhn and Osthus \cite{KuOs12} from $n^{-1/2 + \varepsilon}$ to $\log^4 n / \sqrt{n}$. It remains to determine the correct order of the threshold. For $k \ge 3$ it is known that the threshold is $n^{-1/k}$, while for $k = 1$ (i.e. a Hamilton cycle) the threshold is $\log n / n$. Since our goal was to give an argument that covers all powers of cycles, we believe that doing more precise calculations for the square could slightly improve our upper bound. However, reducing the logarithmic factor significantly or even removing it completely seems to require new ideas. Recently, Bennet, Dudek and Frieze \cite{bennett2016square} announced that $1/\sqrt{n}$ is a threshold in the case $k = 2$. However, their proof is also fully based on a second-moment (akin to Riordan's proof) thus does not have any algorithmic implications.

Another question is whether there exists a polynomial time algorithm for finding Hamilton $k$-cycles in random graphs and tight Hamilton cycles in random hypergraphs for values of $p$ which are close to $n^{-1/k}$ and $1/n$, respectively. Proofs given here provide a randomized quasi-polynomial time algorithm for both problems in case $p$ is a logarithmic factor away from these values, whereas for $p$ being a factor of $n^{\eps}$ away from the threshold polynomial algorithms are known \cite{allen2013tight,KuOs12}. It would be of interest to further improve our results and give polynomial algorithms for the same (or better) range of $p$. Finally, we remark that our algorithms only use randomization at the very beginning to split the given graph into three graphs, in order to mimic the multiple exposure used in the proof. It would be interesting to find a way to avoid this and obtain deterministic algorithms for considered problems.

\bibliographystyle{abbrv}
\bibliography{references}


\end{document}